\documentclass[11pt]{article}

\usepackage{amsmath}
\usepackage{amssymb}
\usepackage{amsthm}
\usepackage{graphicx}
\usepackage{hyperref}
\usepackage{stmaryrd}
\usepackage{multirow}
\usepackage{rotate}

\usepackage[utf8]{inputenc}
\usepackage{fullpage}
\usepackage{dsfont}
\usepackage{listings}
\usepackage{color}
\usepackage{comment}
\usepackage[ruled,linesnumbered,vlined]{algorithm2e}
\usepackage{soul}
\usepackage{subfig}
\usepackage{booktabs}
\usepackage{tikz}
\usepackage{svg}

\newcommand{\R}{\mathds{R}}
\newcommand{\Om}{\Omega}
\newcommand{\om}{\omega}
\newcommand{\bx}{\boldsymbol{x}}
\newcommand{\divv}{\operatorname{div}}
\newcommand{\dx}{\delta x}
\newcommand{\dr}{\delta r}
\newcommand{\tred}{\textcolor{red}}

\newcommand{\vol}{\operatorname{Vol}}
\DeclareMathOperator*{\Minimize}{Minimize}

\DeclareMathOperator*{\cotan}{cotan}

\newcommand{\arc}{\mathcal{A}}
\newcommand{\pairsvw}{\mathbb{A}}
\newcommand{\sa}{\mathcal{S}}

\newcommand{\rood}[1]{}

\SetKwInput{KwInput}{Input}
\SetKwInput{KwOutput}{Output}

\newtheorem{assumption}{Assumption}
\newtheorem{lemma}{Lemma}
\newtheorem{thm}{Theorem}

\theoremstyle{definition}
\newtheorem{definition}{Definition}

\newtheorem{example}{Example}

\makeatletter
\let\orgdescriptionlabel\descriptionlabel
\renewcommand*{\descriptionlabel}[1]{%
  \let\orglabel\label
  \let\label\@gobble
  \phantomsection
  \edef\@currentlabel{#1}%
  \let\label\orglabel
  \orgdescriptionlabel{#1}%
}
\makeatother

\definecolor{dkgreen}{rgb}{0,0.6,0}
\definecolor{gray}{rgb}{0.5,0.5,0.5}
\definecolor{mauve}{rgb}{0.58,0,0.82}

\lstnewenvironment{ttlisting}{
\lstset{
basicstyle=\ttfamily,
frame=none,
numbers=none,
}
}{}

\begin{document}

\title{A Shape-Newton approach to the problem of covering\\ with identical balls\thanks{This work has been partially supported by FAPESP (grants 2013/07375-0, 2016/01860-1, 2018/24293-0, and 2019/25258-7) and CNPq (grants 302682/2019-8, 304258/2018-0, and 408175/2018-4).}}

    
\author{
  Ernesto G. Birgin\thanks{Department of Computer Science, Institute of
    Mathematics and Statistics, University of S\~ao Paulo, Rua do
    Mat\~ao, 1010, Cidade Universit\'aria, 05508-090, S\~ao Paulo, SP,
    Brazil. e-mails: egbirgin@ime.usp.br, rmassambone@ime.usp.br, and ags@ime.usp.br}
  \and
  Antoine Laurain\thanks{Department of Applied Mathematics, Institute of
    Mathematics and Statistics, University of S\~ao Paulo, Rua do
    Mat\~ao, 1010, Cidade Universit\'aria, 05508-090, S\~ao Paulo, SP,
    Brazil. e-mail: laurain@ime.usp.br}
  \and
  Rafael Massambone\footnotemark[2]
  \and
  Arthur G. Santana\footnotemark[2]
}

\date{June 4, 2021}

\maketitle

\begin{abstract}
The problem of covering a region 
of the plane with a fixed number of minimum-radius identical balls is studied in the present work. 
An explicit construction of bi-Lipschitz mappings is provided to model small perturbations of the union of balls.
This allows us to obtain analytical expressions for first- and second-order derivatives using nonsmooth shape optimization techniques under appropriate regularity assumptions. Singular cases are also studied using asymptotic analysis.
For the case of regions given by the union of disjoint convex polygons, algorithms based on Voronoi diagrams that do not rely on approximations are given to compute the derivatives.
Extensive numerical experiments illustrate the capabilities and limitations of the introduced approach. 

\vspace{1mm}

\vspace{1mm}
\end{abstract}

\section{Introduction} \label{intro}

The problem of covering a region of the plane
with a fixed number of minimum-radius identical balls is studied in the present work by expanding the nonsmooth shape optimization approach introduced in \cite{coveringfirst}. 
The main challenge in this previous work was the the first-order shape sensitivity analysis with respect to perturbations of the balls' centers and radii. 
Therefore, investigating the second-order shape sensitivity is a natural albeit challenging extension of \cite{coveringfirst}. 

Shape optimization is the study of optimization problems where the variable is a geometric object; see~\cite{MR2731611,MR3791463,MR1215733}.
One of the key concepts in this discipline is the notion of \textit{shape derivative}, that measures the sensitivity of functions with respect to perturbations of the geometry.
The theoretical study of second-order shape derivatives is a difficult topic in shape optimization. There exists an abundant literature on the shape Hessian in the smooth setting \cite{Dambrine2000,Dambrine2003,MR2731611,MR1215733}; while in the nonsmooth setting it is still an active research topic \cite{MR2148282,LAURAIN2020328}. Numerical methods based on second-order shape derivative are rarely used in shape optimization due to several difficulties. First of all, the second-order shape  derivative is often  difficult to compute and costly to implement numerically, especially when partial differential equations are involved. 
Second, the shape Hessian presents several theoretical issues, such as the {\it two norms-discrepancy} and lack of coercivity, that have been extensively studied in control problems; see \cite{Afraites2008,Dambrine2000} and the references therein. There exist only few attempts at defining numerical methods based on second-order information in shape optimization. In \cite{EH2005}, a regularized shape-Newton method is introduced to solve an inverse problem for star-shaped geometries. Second order preconditioning of the shape gradient has been used in \cite{Hintermller2004} for image segmentation and in \cite{doi:10.1080/10618569908940813,Schmidt2009} for aerodynamic optimization.
Automatic shape differentiation has also been successfully employed to compute first- and second-order shape derivatives \cite{Ham2019,Schmidt2018}.
We also observe that the numerical investigations using Newton-type algorithms \cite{EH2005,Hintermller2004} are set in a relatively smooth setting. 
In \cite{LAURAIN2020328}, the shape Hessian was calculated for nonsmooth geometries and polygons in a form that was convenient  for numerical experiments, but no numerical investigations were performed. 
To the best of our knowledge, the present paper is the first  attempt at designing and analyzing a shape-Newton algorithm in a genuinely nonsmooth setting. 

From a theoretical perspective, the main achievement of \cite{coveringfirst} was to build bi-Lipschitz transformations to model the geometry perturbations corresponding to covering with identical balls. In the present work, these transformations are  key elements for the calculation of the second-order shape derivative, which, unlike the first-order shape derivative, differs from the expression that would be obtained in a smooth setting. Indeed, for the piecewise smooth shapes considered in the covering problem, various terms with a support at  singular boundary points, typically circles intersection,  appear in the shape Hessian.

Due to the generality of the regions to be covered considered in~\cite{coveringfirst}, in the presented numerical experiments, the function that measures the covering and its first-order derivatives were approximated with discretization strategies that may by very time consuming if high precision is required. In the present work, by restricting the region to be covered to be the union of disjoint convex polygons, algorithms based on Voronoi diagrams to compute the covering function and its first- and second-order derivatives \textit{without} relying on approximations are given.

The problem of covering a two-dimensional region with identical balls has already been considered in the literature. Covering equilateral triangles and squares was considered in~\cite{nurmelaT} and~\cite{nurmelaS}, respectively; while covering the union and difference of polygons was considered in~\cite{stoyan}. 
The covering of rectangles, triangles, squares and arbitrary regions was considered in~\cite{heppes}, \cite{melissenT1997}, \cite{melissenS1996} and~\cite{xavier2D}, respectively. 
However, the problem addressed in~\cite{xavier2D} actually consists of covering an arbitrary set of points, which is substantially different from the problem of covering an entire region. All of these papers approach the problem as an optimization problem.  In~\cite{heppes,melissenT1997,melissenS1996} a simulated annealing approach with local search in which the centers of the balls are chosen as points on an adaptive mesh is considered. In~\cite{nurmelaT,nurmelaS}, a discrete rule is used to define the radius; while a BFGS method is used to solve subproblems in which the radius is fixed. A feasible direction method that requires solving a linear programming problem at each iteration was proposed in~\cite{stoyan}. None of the mentioned works addresses the problem in a unified way as a continuous optimization problem, nor do they present first- or second-order derivatives of the functions that define the problem. In~\cite{birgin}, the problem of covering an arbitrary region is modeled as a nonlinear semidefinite programming problem using convex algebraic geometry tools. The introduced model describes the covering problem without resorting to discretizations, but it depends on some polynomials of unknown degrees whose coefficients are difficult to compute, limiting the applicability of the method. 

The rest of this paper is organized as follows. Section~\ref{sec:shape_opt_covering} presents a formal definition of the problem, the formula for the first-order derivative introduced in~\cite{coveringfirst}, and the formula for the second-order derivative being introduced in the present work. Section~\ref{sec:D2G_proof} presents the derivation of the second-order derivatives for non-degenerate cases; while degenerate cases are considered in Section~\ref{sec:singularcases}. Algorithms based on Voronoi diagrams for the exact calculation of the covering function and its first- and second-order derivatives are introduced in Section~\ref{sec:calc_algorithms}. Extensive numerical experiments are given in Section~\ref{sec:num_exp}. Final considerations are given in Section~\ref{sec:conclusion}.\\

\noindent
\textbf{Notation:} Given $x, y \in \R^n$, $x \cdot y = x^\top y \in \R$; while $x \otimes y = x y^\top \in \R^{n \times n}$. 
The divergence of   a sufficiently smooth vector field $\R^2\ni (x,y)\mapsto V(x,y)=(V_1(x,y),V_2(x,y))\in\R^2$ is defined by $\divv V := \frac{\partial V_1}{\partial x} + \frac{\partial V_2}{\partial y}$, and its Jacobian matrix is denoted $DV$.
Given an open set $S\in\R^n$, $\overline{S}$ denotes its closure, $\partial S = \overline{S} \setminus S$  its boundary, and $\vol(S)$ its volume.
Let $B(x_i,r)$ denote an open ball with center $x_i\in\R^2$ and radius $r$.
For a sufficiently smooth set $S\subset\R^2$,  $\nu_S(z)$ denotes the unitary-norm outwards normal vector to $S$ at $z$ and $\tau_S(z)$  the unitary-norm  tangent vector to $\partial S$ at~$z$ (pointing counter-clockwise).
In the particular case $S = B(x_i,r)$ we use the simpler notation $\nu_i:=\nu_{B(x_i,r)}$ and $\tau_i := \tau_{B(x_i,r)}$, and we have
$\nu_i(z) = (\cos \theta_z, \sin \theta_z)^\top$ and $\tau_i(z) = (-\sin \theta_z, \cos \theta_z)^\top$, 
where $\theta_z$ is the angular coordinate of $z-x_i$.
For intersection points $z\in \partial S \cap B(x_i,r)$, we also use the  notation $\nu_{-i}(z): = \nu_S(z)$.

\section{The shape optimization problem}\label{sec:shape_opt_covering}
Let $A\subset \R^2$  and $\Om(\bx,r) = \cup_{i=1}^m B(x_i,r)$ with $\bx:= \{x_i\}_{i=1}^m$. We consider the problem of covering $A$ using  a fixed number $m$ of identical balls $B(x_i, r)$ with minimum radius~$r$, i.e., we are looking for $(\bx, r) \in \R^{2m+1}$ such that $A \subset \Om(\bx, r)$ with minimum~$r$. The problem can be formulated as
\begin{equation} \label{prob1}
\Minimize_{(\bx,r) \in \R^{2m+1}} \; r \; \mbox{ subject to } \; G(\bx,r) = 0,
\end{equation}
where
\begin{align} \label{aquiG}
G(\bx,r) &:=  \vol(A) - \vol(A \cap \Om(\bx,r)).
\end{align}
Note that $G(\bx,r) = 0$ if and only if $A\subset \Om(\bx,r)$ up to a set of zero measure, i.e., when $ \Om(\bx,r)$ covers $A$.

The derivatives of~$G$ can be computed using techniques of shape calculus \cite{MR2731611,MR3791463,MR3535238,LAURAIN2020328,MR1215733}. In particular it was shown in~\cite{coveringfirst} that, under suitable assumptions,
\begin{align}\label{gradG}
\nabla G(\bx,r)
= -
\begin{pmatrix}
\displaystyle  \int_{\arc_1} \nu_1(z) \, dz,
\cdots,
\displaystyle  \int_{\arc_m} \nu_m(z) \, dz,
\displaystyle  \int_{\partial \Om(\bx,r)\cap A} \, dz
\end{pmatrix}^\top,
\end{align}
where
\begin{equation} \label{partition}
\mathcal{A}_i = \partial B(x_i,r) \cap \partial \Om(\bx,r) \cap A
\end{equation}
for $i=1,\dots,m$.

In the present work, we show that
\begin{equation} \label{derivadasegunda}
\nabla^2 G (\bx,r) = 
\begin{pmatrix}
\nabla^2_{\bx} G (\bx,r) & \nabla^2_{\bx,r} G (\bx,r)\\
\nabla^2_{\bx,r} G (\bx,r)^\top & \nabla^2_r G (\bx,r)
\end{pmatrix},
\end{equation}
where $\nabla^2_{\bx} G (\bx,r) \in \R^{2m \times 2m}$, $\nabla^2_{\bx,r} G (\bx,r) \in \R^{2m}$, and $\nabla^2_r G (\bx,r)= \partial_{r}^{2} G(\bx,r) \in \R$ are described below. Their description is based on the fact that each set $\arc_i$ can be represented by a finite number $m_i \geq 0$ of arcs of the circle $\partial B(x_i,r)$.
Note that, since $(\cup_{i=1}^m \partial B(x_i,r)) \cap \partial \Om(\bx,r) = \partial \Om(\bx,r)$, by~\eqref{partition}, 
\begin{equation} \label{unionAi}
\bigcup_{i=1}^m \arc_i = \partial\Om(\bx,r)\cap A,
\end{equation}
i.e., the union of all $\arc_i$ represents a partition of $\partial\Om(\bx,r)\cap A$; see Figure~\ref{fig:arcos}.
Each arc in $\mathcal{A}_i$ can be represented by a pair of points $(v,w)$, named starting and ending points, in counter-clockwise direction, i.e., such that the angular coordinates $\theta_v$ and $\theta_w$ of $v-x_i$ and $w-x_i$, respectively, satisfy $\theta_v \in [0,2\pi)$ and $\theta_w \in (\theta_v,\theta_v+2\pi]$; see Figure~\ref{fig:sec2}.
If $\arc_i$ is not a full circle, we denote by $\pairsvw_i$ the set of pairs $(v,w)$ that represent the arcs in $\arc_i$; otherwise, we define $\pairsvw_i=\emptyset$. In addition, if $\arc_i$ is a full circle, then we set $\mathrm{Circle}(\pairsvw_i)$ equal to true; otherwise, we set $\mathrm{Circle}(\pairsvw_i)$ equal to false. We say a configuration $(\bx,r)$ is non-degenerate if, for every $i=1,\dots,m$, every $(v,w) \in \pairsvw_i$, and every $z \in \{v,w\}$,  there exists one and only one $\nu_{-i}(z)$ and $\nu_{-i}(z) \cdot \tau_i(z) \neq 0$. A characterization of non-degenerate configurations, which satisfy Assumptions~\ref{a1} and~\ref{a3}, is given in the next section.



\begin{figure}[ht!]
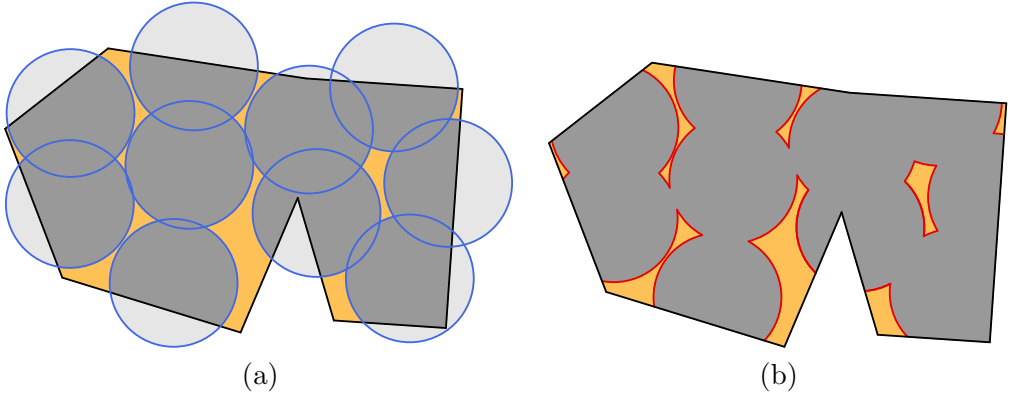

\begin{center}
\begin{tabular}{cc}
\includegraphics{A_omega.mps} & \includegraphics{A_partialOmega.mps} \\
(a) & (b)
\end{tabular}
\end{center}
\caption{(a) represents a region $A$ to be covered and an arbitrary configuration of balls $\Om(\bx,r)$. (b) represents, in red, $\partial \Om(\bx,r) \cap A$. Each $\mathcal{A}_i$ corresponds to the red arcs that intersect $\partial B(x_i,r)$. Note that, in this example, most sets $\arc_i$ contain two or three maximal arcs; and there is only one set $\arc_i$ with four maximal arcs.} 
\label{fig:arcos}
\end{figure}

Assuming~$(\bx,r)$ is non-degenerate, we have that $\nabla^2_r G(\bx,r)$ in~\eqref{derivadasegunda} is given by
\begin{equation} \label{d2Grr}
\nabla^2_r G(\bx,r) = 
- \frac{\mathrm{Per}(\partial\Om(\bx,r)\cap A)}{r}
- \sum_{i=1}^m  \sum_{(v,w) \in \pairsvw_i} \left\llbracket \, \frac{|L(z)| - \nu_{-i}(z) \cdot \nu_i(z)}{\nu_{-i}(z) \cdot \tau_i(z)} \, \right\rrbracket_v^w,
\end{equation}
where, for an arbitrary expression $\Phi(z)$, $\llbracket \Phi(z) \rrbracket_v^w := \Phi(w) - \Phi(v)$, $\mathrm{Per}(S)$ denotes the perimeter of the set~$S$, and, for an extreme $z$ of an arc represented by $(v,w) \in \pairsvw_i$, $L(z)=\{ \ell \in \{1,\dots,m\} \setminus \{i\} \;|\; z \in \partial B(x_\ell,r) \}$. 

Matrix $\nabla^2_{\bx} G (\bx,r)$ in~\eqref{derivadasegunda} is given by the $2 \times 2$ diagonal blocks
\begin{equation} \label{d2xixiG}
\partial^2_{x_i x_i} G(\bx,r) =
\frac{1}{r} \int_{\arc_i} - \nu_i(z) \otimes \nu_i(z) + \tau_i(z) \otimes \tau_i(z) \, dz + \sum_{(v,w) \in \pairsvw_i} \left\llbracket \, \frac{\nu_{-i}(z) \cdot \nu_i(z)}{\nu_{-i}(z) \cdot \tau_i(z)} \, \nu_i(z) \otimes \nu_i(z) \, \right\rrbracket_{v}^{w}
\end{equation}
and the $2 \times 2$ off-diagonal blocks
\begin{equation} \label{d2Gxixj}
\partial^2_{x_i x_\ell} G(\bx,r) = 
\sum_{v \in \mathcal{I}_{i\ell}} \frac{\nu_i(v) \otimes \nu_\ell(v)}{\nu_\ell(v) \cdot \tau_i(v)} - \sum_{w \in \mathcal{O}_{i\ell}} \frac{\nu_i(w) \otimes \nu_\ell(w)}{\nu_\ell(w) \cdot \tau_i(w)},
\end{equation}
where ${\cal I}_{i\ell} = \{ v \in \partial B(x_\ell,r) \; | \; (v,\cdot) \in \pairsvw_i \}$ and $\mathcal{O}_{i\ell} = \{ w \in \partial B(x_\ell,r) \; | \; (\cdot,w) \in \pairsvw_i \}$. (Note that ${\cal I}_{i\ell}=\mathcal{O}_{i\ell}=\emptyset$ for all $\ell\neq i$ if $\pairsvw_i=\emptyset$.)
Finally, array $\nabla^2_{\bx,r} G (\bx,r)$ in~\eqref{derivadasegunda} is given by the 2-dimensional arrays 
\begin{equation} \label{eq:781}
\partial^2_{x_i r} G(\bx,r) = 
- \frac{1}{r} \int_{\arc_i} \nu_i(z) \, dz
+ \sum_{(v,w) \in \pairsvw_i} \left\llbracket \, \frac{\nu_{-i}(z) \cdot \nu_i(z)}{\nu_{-i}(z) \cdot \tau_i(z)} \nu_i(z) - \sum_{\ell \in L(z)} \frac{\nu_i(z)}{\tau_i(z) \cdot \nu_{\ell}(z)} \, \right\rrbracket_v^w.
\end{equation}

\begin{figure}
\centering
\includegraphics{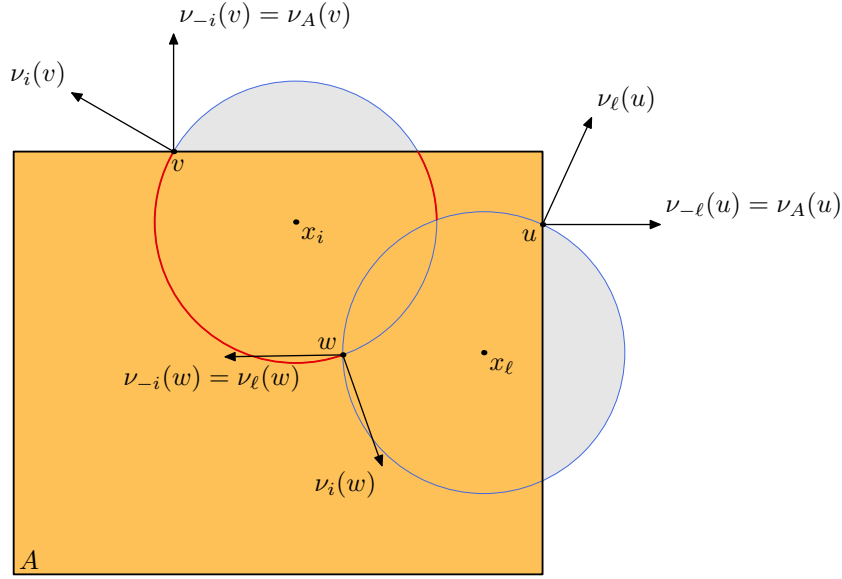}
\caption{The set $\arc_i = \partial B(x_i,r) \cap \Omega(\bx,r) \cap A$ is composed of two arcs (in red). If $z \in \partial B(x_i,r) \cap \partial B(x_{\ell},r)$ for some $\ell \neq i$, as for $z=w$, then $\nu_{-i}(z) = \nu_{\ell}(z)$, while if $z \in \partial B(x_i,r) \cap \partial A$, as for $z \in \{u,v\}$, then $\nu_{-i}(z) = \nu_{A}(z)$.}
\label{fig:sec2}
\end{figure}

\section{Proof of second-order differentiability of \texorpdfstring{$G$}{G}} \label{sec:D2G_proof}

In this section, we prove that the second-order derivatives of $G$, as defined in~\eqref{aquiG}, are given by (\ref{derivadasegunda}, \ref{d2Grr}, \ref{d2xixiG}, \ref{d2Gxixj}, \ref{eq:781}). 
In \cite{coveringfirst} we have built appropriate bi-Lipschitz  mappings $T_t$ in order to use integration by substitution for the differentiation of $G(\bx+t\delta\bx,r)$ and $G(\bx,r+t\dr)$. Some of the more technical aspects of these constructions were related to the fact that $G(\bx,r)$ is an area functional, which required defining $T_t$ on $\Om(\bx,r)\cap A$ and on $\partial(\Om(\bx,r)\cap A)$. 
Since $\nabla G$ only involves boundary integrals that in addition can be decomposed into integrals on arcs, this facilitates the construction of the mappings $T_t$ required for the calculation of $\nabla^2 G (\bx,r)$, as $T_t$ only needs to be defined on $\partial \Om(\bx,r)\cap A$.

We consider two  types of transformations for the shape sensitivity analysis.
First, in the case of fixed radius and center perturbations one needs a mapping $T_t$ between the reference set $\partial\Om(\bx,r)\cap A$ and the perturbed set $\partial\Om(\bx +t\delta \bx,r) \cap A$; see Theorem~\ref{thm:EB}.
Second, in the case of fixed centers and radius perturbation one needs a mapping $T_t$ between the reference set $\partial\Om(\bx, r)\cap A$ and the perturbed set  $\partial\Om(\bx, r+t\delta r)\cap A$; see Theorem~\ref{thm5}.
The shape sensitivity analysis of $\nabla G$ is then achieved through integration by substitution using $T_t$.
The construction of these mappings $T_t$ is similar to the constructions in \cite{coveringfirst}; however the results are presented in a different way as we need specific properties of $T_t$ to compute the derivatives of $\nabla G$. 
One of the main differences with respect to~\cite{coveringfirst} appears in Theorem~\ref{thm:EB}, where one considers a simultaneous perturbations of all the balls' center, which allows us to simplify the calculations of the Hessian of $G$. 
On the one hand,  $T_t$ was used in \cite{coveringfirst} mainly to prove first-order shape differentiability and its unusual structure did not affect the expression of the first-order shape derivative, in the sense that a similar formula would have been obtained in a smooth setting.   
On the other hand, the expression of the second-order shape derivative of~$G$ at a nonsmooth reference domain $\Om$ differs significantly    from the expression that would be obtained for a smooth $\Om$, as it involves terms with a support at singular boundary points of $\Om$, and the particular structure of $T_t$ now plays an important role in the calculation of those singular terms. 
This can be understood by considering that, unlike the first-order derivative,  the second-order shape derivative depends on the tangential component of $\left. \partial_t T_t\right|_{t= 0}$ on the boundary of the reference domain.

In \cite{coveringfirst}, we have described detailed conditions to avoid degenerate situations and we also discussed various examples of such degeneracies and how they may affect the numerical algorithm. 
In the present paper we use the same conditions  to prove second-order differentiability of $G$. To summarize, the main issues when studying the differentiability of $G$ arise when two balls are tangent or exactly superposed, when the boundaries of more than two balls intersect at the same point, or when $\Om(\bx,r)$ and $A$ are not compatible in the sense of Definition \ref{def:compatible}.
The role of Assumptions~\ref{a1} and \ref{a3}  is to avoid these singular cases, which allows us to prove second-order differentiability of $G$.
We emphasize that these assumptions  only exclude a null-measure set of balls' configurations in $\R^{2m+1}$, and in Section \ref{sec:singularcases} we show via the study of several singular cases that the second-order differentiability of~$G$ fails when these assumptions are not satisfied. 

\begin{assumption}\label{a1}
The centers $\{x_i\}_{i=1}^m$ satisfy $\|x_i-x_j\|\notin \{0,2r\}$ for all $1\leq i,j\leq m$, $i\neq j$ and $\partial B(x_i,r)\cap \partial B(x_j,r) \cap \partial B(x_k,r) =\emptyset$ for all $1\leq i,j,k\leq m$ with $i,j,k$ pairwise distinct.
\end{assumption}
\begin{definition}\label{def:compatible}
Let $\om_1,\om_2$ be open subsets of $\R^2$. 
We call $\om_1$ and $\om_2$ {\it compatible} if $\om_1\cap\om_2\neq\emptyset$, $\om_1$ and $\om_2$ are Lipschitz domains, and the following conditions hold:
(i) $\om_1\cap\om_2$ is a Lipschitz domain;
(ii) $\partial\om_1\cap \partial\om_2$ is finite;
(iii) $\partial\om_1$ and $\partial\om_2$ are locally smooth in a neighborhood of $\partial\om_1\cap \partial\om_2$;
(iv) $\tau_1(x)\cdot \nu_2(x)\neq 0$ for all $x\in \partial\om_1\cap \partial\om_2$, where $\tau_1(x)$ is a tangent vector to $\partial\om_1$ at $x$ and $\nu_2(x)$ is a normal vector to $\partial\om_2$ at $x$.
\end{definition}
\begin{assumption}\label{a3}
Sets  $\Om(\bx,r)$ and $A$ are compatible.
\end{assumption}
We observe that $\Om(\bx,r)$ is Lipschitz under Assumption~\ref{a1}, and if, in addition, the intersection of $\partial\Om(\bx,r)$ and $\partial A$ is empty, then Assumption~\ref{a3} holds. 
Hence, in this particular case we can drop  Assumption~\ref{a3} in Theorems~\ref{thm:EB} and~\ref{thm5}.

We also recall the following basic results, which are  key ingredients for the calculation of the shape Hessian of $G$.
\begin{thm}[Tangential divergence theorem]\label{lem:tandivpol}
Let $\Gamma\subset\R^2$ be a $C^k$ open curve, $k\geq 2$, with a parameterization $\gamma$, and denote $(v,w)$ the starting and ending points of $\Gamma$, respectively, with respect to $\gamma$.
Let $\tau$ be the unitary-norm  tangent vector to $\Gamma$, $\nu$ the unitary-norm normal vector to $\Gamma$, and $\mathcal{H}$ the mean curvature of $\Gamma$, with respect to the parameterization $\gamma$.
Let $F\in W^{1,1}(\Gamma,\R^2)\cap C^0(\overline{\Gamma},\R^2)$,
then we have 
$$ \int_{\Gamma} \divv_\Gamma(F) 
= \int_{\Gamma} \mathcal{H} F\cdot \nu +  F(w)\cdot \tau(w) - F(v) \cdot \tau(v)
= \int_{\Gamma} \mathcal{H} F\cdot \nu +  \llbracket F(z)\cdot \tau(z)\rrbracket_v^w,$$
where  $\divv_\Gamma(F):=\divv(F) - DF\nu\cdot\nu$ is the tangential divergence of $F$ on $\Gamma$.
\end{thm}
\begin{proof}
The result follows from  \cite[\S~7.2]{MR756417} and \cite[Ch.~9, \S~5.5]{MR2731611}.
\end{proof}

\begin{lemma}[Integration by substitution for line integrals]\label{lem:intsub}
Let $\Gamma\subset\R^2$ be a $C^k$ open curve, $k\geq 2$, and  $\nu$ a unitary-norm normal vector to $\Gamma$.
Let $F\in C^0(\overline{\Gamma},\R^2)$ and $T_t:\overline{\Gamma}\to T_t(\overline{\Gamma})$ be a bi-Lipschitz mapping. 
Then  
$$ \int_{T_t(\Gamma)} F(z)\, dz 
= \int_{\Gamma} F(T_t(z)) \omega_t,
$$
where 
\begin{equation}\label{density:changevar}
\omega_t(z) := \|M(z,t)\nu(z)\|
\end{equation}
and $M(z,t) := \det(DT_t(z)) DT_t(z)^{-\top}$ is the cofactor matrix of $DT_t(z)$.
Furthermore, we have 
\begin{equation}\label{density:der}
\partial_t\omega_t|_{t=0} = \divv_\Gamma V \text{ with } V: = \partial_t T_t|_{t=0} \text{ on } \Gamma.
\end{equation}
\end{lemma}
\begin{proof}
See \cite[Prop.~5.4.3]{MR3791463}.
\end{proof}

\subsection{Construction of  a perturbation field for center perturbations}

Theorem \ref{thm:EB} below employs several ideas from \cite[Thm. 3.2 \& Thm. 3.6]{coveringfirst}.
However, an important difference is that we consider  simultaneous center perturbations for all balls instead of just one, which is more convenient for the calculation of $\nabla^2 G$.
Theorem \ref{thm:EB} provides an appropriate mapping $T_t$ for the differentiation of $\partial_{x_i} G (\bx+t\delta\bx,r)$ that will be used in Sections~\ref{sec:second_D_xiG} and~\ref{sec:derxixl} and for the differentiation of $\partial_{r} G (\bx+t\delta\bx,r)$ in Section~\ref{sec:derxir}.

\begin{thm}\label{thm:EB}
Suppose that Assumptions~\ref{a1} and \ref{a3} hold.
Then there exists $t_0>0$ such that  for all $t\in [0,t_0]$ we have the following  decomposition
\begin{equation}\label{eq:dec_E}
\partial \Om(\bx+t\delta\bx,r)\cap A = \bigcup_{k=1}^{\bar{k}} \sa_k(t),  
\end{equation}
where $\bar{k}$ is independent of $t$, 
 $\sa_k(t)$ are arcs parameterized by an angle aperture $[\theta_{k,v}(t), \theta_{k,w}(t) ]$, and $t\mapsto \theta_{k,v}(t)$,  $t\mapsto  \theta_{k,w}(t)$ are continuous functions on $[0,t_0]$.

Also, for all $t\in [0,t_0]$  there exists a bi-Lipschitz mapping $T_t:  \partial\Om(\bx,r)\cap A\to \R^2$ satisfying $T_t(\partial\Om(\bx,r)\cap A) = \partial\Om(\bx+t\delta\bx,r)\cap A $ and $T_t(\sa_k(0))=\sa_k(t)$ for all $k=1,\dots,\bar{k}$.
Furthermore, we have
\begin{equation}\label{eq:Vtheta_arc}
V := \left.\partial_t T_t\right|_{t= 0} 
=\dx_i +  \partial_t\xi(0,\theta) r \tau_i \text{ on } \sa_k(0)\subset \partial B(x_i,r),
\end{equation}
where $\xi$ is defined in \eqref{114} and
\begin{align}\label{eq:Vthetaa}
V(z) & = \dx_i  - \frac{\nu_A(z)\cdot \dx_i}{ \tau_i(z) \cdot \nu_A(z)}\tau_i(z)
\quad  \text{ if } z\in\partial B(x_i,r)\cap \partial A,\\
\label{eq:Vthetaa_2}
V(z) & = \dx_i  - \frac{\nu_\ell(z)\cdot (\dx_i-\dx_\ell)}{ \tau_i(z) \cdot \nu_\ell(z)}\tau_i(z)
\quad  \text{ if } z\in\partial B(x_i,r)\cap \partial B(x_\ell,r), i\neq\ell.
\end{align}
\end{thm}
\begin{proof} 
The decomposition \eqref{eq:dec_E} relies on Assumptions~\ref{a1} and \ref{a3} and is obtained in a similar way as in \cite[Thm.~3.2]{coveringfirst}. 
Therefore, in this proof we focus on the construction of the mapping $T_t$. We observe that each extremity of the arcs $\sa_k(t)$ in the decomposition \eqref{eq:dec_E} is either a point belonging to $\partial B(x_i+t\dx_i,r)\cap \partial A$ or a point in $\partial B(x_i+t\dx_i,r)\cap\partial B(x_\ell+t\dx_\ell,r)$.

We first provide a general formula for the angle $\vartheta(t)$, in local polar coordinates with the pole $x_i+t\dx_i$, describing  an intersection point between the circle $\partial B(x_i+t\dx_i,r)$ and $\partial A$.
Let $z\in\partial B(x_i,r)\cap \partial A$ and $\phi$ be the oriented distance function to $A$, defined as 
$ \phi(x) := d(x,A) - d(x,A^c)$, where  $d(x,A)$ is the distance from $x$ to the set $A$.
Since $\Om(\bx,r)$ and $A$ are compatible due to Assumption~\ref{a3}, it follows that $\partial A$ is locally smooth around the points $\partial B(x_i,r)\cap \partial A$, hence there exists a neighborhood $U_z$ of $z$ such that the restriction of $\phi$ to $U_z$ is smooth, $\phi(x) = 0$ and  $\|\nabla \phi(x)\| = 1$ for all $x\in\partial A \cap U_z$.  

Let $(r,\theta_z)$ denote the polar coordinates of $z$, with the pole $x_i$.
Introduce the function
$$\psi(t,\vartheta) =  \phi\left(   x_i +t\delta x_i + r \begin{pmatrix} \cos\vartheta \\ \sin\vartheta \end{pmatrix}\right) .$$
We compute 
$$\partial_{\vartheta} \psi(0,\theta_z) = r\begin{pmatrix} -\sin\theta_z \\ \cos\theta_z \end{pmatrix}
\cdot 
\nabla  \phi\left( x_i +  r \begin{pmatrix} \cos\theta_z \\ \sin\theta_z \end{pmatrix}\right)
= r \tau_i(z) \cdot \nabla  \phi(z).$$
Since   $\Om(\bx,r)$ and $A$ are compatible, $B(x_i,r)$ is not tangent to $\partial A$ and using  $\|\nabla \phi(z)\| = 1$ we obtain
$\tau_i(z) \cdot \nabla  \phi(z)  \neq 0$.
Thus, we can apply the implicit function theorem and this yields the existence of a smooth function $[0,t_0]\ni t\mapsto \vartheta(t)$ with $\psi(t,\vartheta(t)) = 0$ and $\vartheta(0) = \theta_z$.
We also compute, using $\nabla  \phi(z) = \|\nabla  \phi(z)\| \nu_A(z)$ since $\phi$ is the oriented distance function  to $\partial A$,
\begin{equation}\label{der_vartheta0}
\vartheta'(0) = -\frac{\partial_t \psi(0, \vartheta(0))}{\partial_\vartheta \psi(0, \vartheta(0))}
= - \frac{\nabla \phi(z)\cdot \dx_i}{r \tau_i(z) \cdot \nabla  \phi(z)} 
= - \frac{\nu_A(z)\cdot \delta x_i}{r \tau_i(z) \cdot \nu_A(z)}.
\end{equation}

We now consider the second case of an intersection point in $\partial B(x_i+t\dx_i,r)\cap\partial B(x_j+t\dx_\ell,r)$, $i\neq \ell$.
Introduce the functions
$$\psi(t,\vartheta) =  \|\zeta(t,\vartheta)\|^2  -  r^2 \ \text{ with }\ \zeta(t,\vartheta) = x_i +t\dx_i - x_\ell -t\dx_\ell + r  \begin{pmatrix} \cos\vartheta \\ \sin\vartheta \end{pmatrix}.$$
Observe that $\vartheta \mapsto \zeta(t,\vartheta)$ is a parameterization of the circle $\partial B(x_i+t\dx_i,r)$ in a coordinate system of center $x_\ell$, which means that the solutions of the equation $\psi(t,\vartheta)=0$ describe the intersections between  $\partial B(x_i+t\dx_i,r)$ and  $\partial B(x_\ell+t\dx_\ell,r)$.
We compute
$\partial_{\vartheta} \psi(0,\vartheta) =  2 \zeta(0,\vartheta) \cdot \partial_{\vartheta}  \zeta(0,\vartheta) $
with 
\begin{align*}
\zeta(0,\vartheta) &= x_i - x_\ell + r \begin{pmatrix} \cos\vartheta \\ \sin\vartheta \end{pmatrix} \quad
\text{ and }\quad\partial_{\vartheta} \zeta(0,\vartheta) =r \begin{pmatrix} -\sin\vartheta \\ \cos\vartheta \end{pmatrix}.
\end{align*}
Now let $z\in\partial B(x_i,r)\cap\partial B(x_\ell,r)$ and 
let $\theta_z$ be the corresponding angle in a polar coordinate system with pole $x_i$.
Since Assumption~\ref{a1} is satisfied, it is easy to see that
\begin{equation*}
\partial_{\vartheta} \psi(0,\theta_z) =  2 \zeta(0,\theta_z) \cdot \partial_{\vartheta}  \zeta(0,\theta_z) \neq 0. 
\end{equation*}
Hence, the implicit function theorem can be applied to  $(t,\vartheta) \mapsto \psi(t,\vartheta)$ in a neighbourhood of $(0,\theta_z)$. 
This yields the existence, for $t_0$ sufficiently small, of a smooth function $t\mapsto\vartheta(t)$ in $[0,t_0]$ such that $\psi(t,\vartheta(t)) = 0$ in $[0,t_0]$ and $\vartheta(0) = \theta_z$.
We also have the derivative
\begin{equation*}
\vartheta'(t) = -\frac{\partial_t \psi(t,\vartheta(t))}{\partial_{\vartheta} \psi(t,\vartheta(t))}
= -\frac{ \zeta(t,\vartheta(t)) \cdot \partial_t  \zeta(t,\vartheta(t)) }{  \zeta(t,\vartheta(t)) \cdot \partial_{\vartheta}  \zeta(t,\vartheta(t)) },
\end{equation*}
and in particular, using $\nu_i = (\cos\theta_z,  \sin\theta_z)^\top$ and $\tau_i = (-\sin\theta_z,  \cos\theta_z)^\top$,
\begin{equation}\label{theta_prime_a_0}
\vartheta'(0)
= -\frac{ ( x_i - x_\ell + r \nu_i)\cdot (\dx_i -\dx_\ell)}{ ( x_i - x_\ell + r\nu_i)\cdot (r \tau_i)}
= -\frac{ \nu_\ell\cdot (\dx_i - \dx_\ell)}{ r  \nu_\ell\cdot \tau_i}.
\end{equation}

We are now ready to build the mapping $T_t$.
Let $\sa(t)\subset \partial B(x_i+t\dx_i,r)$ be one of the arcs  parameterized by the angle aperture $[\theta_v(t), \theta_w(t) ]$ in the decomposition~\eqref{eq:dec_E}; we have dropped the index $k$ for simplicity. 
Then, $\theta_v(t)$ and $\theta_w(t)$ are given by $\vartheta(t)$ with either  $\theta_z=\theta_v(0)$ or $\theta_z=\theta_w(0)$, and $\vartheta(t)$ either corresponds to an intersection $\partial B(x_i+t\dx_i,r)\cap \partial A$ or to an intersection $\partial B(x_i+t\dx_i,r)\cap\partial B(x_\ell+t\dx_\ell,r)$. 
Thus we define $T_t$ on the arc $\sa(0)$ as
\begin{equation}\label{113}
T_t(x) :=  x_i + t\dx_i +r  \begin{pmatrix} \cos\xi(t,\theta)  \\ \sin\xi(t,\theta) \end{pmatrix} 
 \text{ with }
 x = x_i  +r \begin{pmatrix} \cos\theta \\ \sin\theta \end{pmatrix}
 \in \sa(0) ,
 \end{equation}
where
\begin{equation}\label{114}
\resizebox{0.9\textwidth}{!}{$
 \xi(t,\theta) := \alpha(t) (\theta - \theta_w(0)) + \theta_w(t) \text{ for } (t,\theta)\in [0,t_0]\times [\theta_v(0),\theta_w(0)] \text{ and }
 \alpha(t) := \frac{\theta_w(t) - \theta_v(t)}{\theta_w(0) - \theta_v(0)}.
 $}
\end{equation}
The bi-Lipschitz property of $T_t$ on $\partial\Om(\bx,r)\cap A$ is obtained as in the proof of \cite[Thm.~3.3]{coveringfirst}.

Finally, differentiating in \eqref{113} with respect to $t$ and using $\xi(0,\theta)=\theta$ we get \eqref{eq:Vtheta_arc}.
Then \eqref{114} yields  $\xi(t,\theta_v(0)) =  \theta_v(t)$, $\xi(t,\theta_w(0)) =  \theta_w(t)$, $\partial_t\xi(0,\theta_v(0)) = \theta'_a(0)$, $ \partial_t\xi(0,\theta_w(0)) = \theta'_b(0)$, consequently  using \eqref{der_vartheta0} we obtain \eqref{eq:Vthetaa}  and using \eqref{theta_prime_a_0} we obtain \eqref{eq:Vthetaa_2}. 
\end{proof}

\subsection{Construction of  a perturbation field for radius perturbations}\label{sec:der_rr}

Theorem~\ref{thm5} below relies on several ideas from \cite[Thm. 3.3 \& Thm. 3.8]{coveringfirst}, and provides an appropriate mapping $T_t$ for the differentiation of $\partial_{r} G (\bx,r+t\dr)$ that will be used in Section~\ref{sec:second_D_rG_fast}. 
\begin{thm}\label{thm5}
Suppose that Assumptions~\ref{a1} and \ref{a3} hold.
Then there exists $t_0>0$ such that  for all $t\in [0,t_0]$ we have the following  decomposition
\begin{equation}\label{eq:dec_Eradius}
\partial \Om(\bx,r+t\dr)\cap A = \bigcup_{k=1}^{\bar{k}} \sa_k(t),  
\end{equation}
where $\bar{k}$ is independent of $t$, $\sa_k(t)$ are arcs parameterized by an angle aperture $[\theta_{k,v}(t), \theta_{k,w}(t) ]$, and $t\mapsto \theta_{k,v}(t)$,  $t\mapsto  \theta_{k,w}(t)$ are continuous functions on $[0,t_0]$.

Also, for all $t\in [0,t_0]$  there exists a bi-Lipschitz mapping $T_t:  \partial\Om(\bx,r)\cap A\to \R^2$ satisfying $T_t(\Om(\bx,r)\cap A) = \partial\Om(\bx,r+t\dr)\cap A $ and $T_t(\sa_k(0))=\sa_k(t)$ for all $k=1,\dots,\bar{k}$.
In addition, we have
\begin{equation}\label{eq:Vtheta_arc_radius}
V := \left.\partial_t T_t\right|_{t= 0} 
=\dr \nu_i +  \partial_t\xi(0,\theta) r \tau_i \text{ on } \sa_k(0)\subset \partial B(x_i,r),
\end{equation}
where $\xi$ is defined in \eqref{114} and
\begin{align}\label{eq:Vthetaa_radius}
V(z) & = \dr \nu_i(z)  - \dr\frac{\nu_A(z)\cdot \nu_i(z)}{ \tau_i(z) \cdot \nu_A(z)}\tau_i(z)
\quad  \text{ if } z\in\partial B(x_i,r)\cap \partial A,\\
\label{eq:Vthetaa_2_radius}
V(z) & = \dr \nu_i(z)   + \dr \frac{ 1 - \nu_\ell(z)\cdot\nu_i(z)}{ \tau_i(z) \cdot \nu_\ell(z)}\tau_i(z)
\quad  \text{ if } z\in\partial B(x_i,r)\cap \partial B(x_\ell,r), i\neq\ell.
\end{align}
\end{thm}
\begin{proof}
The proof has the same structure as the proof of Theorem \ref{thm:EB}, i.e., we separate the two cases of a point belonging to $\partial B(x_i,r+t\dr)\cap \partial A$ and a point in $\partial B(x_i,r+t\dr)\cap\partial B(x_\ell,r+t\dr)$.
The decomposition \eqref{eq:dec_Eradius} relies on Assumptions~\ref{a1} and \ref{a3} and is obtained in a similar way as in \cite[Thm.~3.2]{coveringfirst}. 

First we consider the case of a point in $\partial B(x_i,r+t\dr)\cap \partial A$.
We provide a general formula for the angle $\vartheta(t)$, in local polar coordinates with pole $x_i$, describing  such an intersection point. Let $z\in\partial B(x_i,r)\cap \partial A$ and $(r,\theta_z)$ denote the polar coordinates of $z$ with center $x_i$. Let  $\phi$ be the oriented distance function to $A$ defined as in the proof of Theorem \ref{thm:EB}.
Introduce the function
$$\psi(t,\vartheta) =  \phi\left(   x_i + (r+t\dr) \begin{pmatrix} \cos\vartheta \\ \sin\vartheta \end{pmatrix}\right) .$$
We compute 
$$\partial_{\vartheta} \psi(0,\theta_z) = r\begin{pmatrix} -\sin\theta_z \\ \cos\theta_z \end{pmatrix}
\cdot 
\nabla  \phi\left( x_i +  r \begin{pmatrix} \cos\theta_z \\ \sin\theta_z \end{pmatrix}\right)
= r \tau_i(z) \cdot \nabla  \phi(z).$$
Since   $\Om(\bx,r)$ and $A$ are compatible due to Assumption~\ref{a3}, $B(x_i,r)$ is not tangent to $\partial A$ and using  $\|\nabla \phi(z)\| = 1$ we obtain
$\tau_i(z) \cdot \nabla  \phi(z)  \neq 0$.
Thus, we can apply the implicit function theorem and this yields the existence of a smooth function $[0,t_0]\ni t\mapsto \vartheta(t)$ with $\psi(t,\vartheta(t)) = 0$ and $\vartheta(0) = \theta_z$.
We also compute the derivative
\begin{equation}\label{der_vartheta}
\vartheta'(0) = -\frac{\partial_t \psi(0, \vartheta(0))}{\partial_\vartheta \psi(0, \vartheta(0))}
=  -\frac{\partial_t \psi(0, \vartheta(0))}{\partial_\vartheta \psi(0, \vartheta(0))}
= - \frac{\dr \nabla \phi(z)\cdot \nu_i(z)}{r \tau_i(z) \cdot \nabla  \phi(z)} 
= - \frac{\dr\nu_A(z)\cdot \nu_i(z)}{r \tau_i(z) \cdot \nu_A(z)},
\end{equation}
where we have used $\nabla  \phi(z) = \|\nabla  \phi(z)\| \nu_A(z)$ since $\phi$ is the oriented distance function  to $\partial A$.

Now we provide a general formula for the angle $\vartheta(t)$, in local polar coordinates with pole $x_i$, describing  an intersection point of two circles $\partial B(x_i,r + t\dr)$ and  $\partial B(x_\ell,r + t\dr)$, $i\neq \ell$.
Introduce 
$$\psi(t,\vartheta) =  \|\zeta(t,\vartheta)\|^2  -  (r+t\dr)^2 \ \text{ with }\ \zeta(t,\vartheta) = x_i - x_\ell + (r + t\dr) \begin{pmatrix} \cos\vartheta \\ \sin\vartheta \end{pmatrix}.$$
Observe that $\vartheta \mapsto \zeta(t,\vartheta)$ is a parameterization of the circle $\partial B(x_i,r + t\dr)$ in a coordinate system of center $x_\ell$, which means that the solutions of  $\psi(t,\vartheta)=0$ describe the intersections between  $\partial B(x_i,r + t\dr)$ and  $\partial B(x_\ell,r + t\dr)$.
We compute
$\partial_{\vartheta} \psi(0,\vartheta) =  2 \zeta(0,\vartheta) \cdot \partial_{\vartheta}  \zeta(0,\vartheta) $
with 
\begin{align*}
\zeta(0,\vartheta) &= x_i - x_\ell + r \begin{pmatrix} \cos\vartheta \\ \sin\vartheta \end{pmatrix} \quad
\text{ and }\quad\partial_{\vartheta} \zeta(0,\vartheta) =r \begin{pmatrix} -\sin\vartheta \\ \cos\vartheta \end{pmatrix}.
\end{align*}
Now let $z\in\partial B(x_i,r)\cap\partial B(x_\ell,r)$ and 
let $\theta_z$ be the corresponding angle in a polar coordinate system with pole $x_i$.
Since the conditions of Assumption~\ref{a1} hold, it is easy to see that
\begin{equation*}
\partial_{\vartheta} \psi(0,\theta_z) =  2 \zeta(0,\theta_z) \cdot \partial_{\vartheta}  \zeta(0,\theta_z)  \neq 0. 
\end{equation*}
Hence, the implicit function theorem can be applied to  $(t,\vartheta) \mapsto \psi(t,\vartheta)$ in a neighbourhood of $(0,\theta_z)$. 
This yields the existence, for $t_0$ sufficiently small, of a smooth function $t\mapsto\vartheta(t)$ in $[0,t_0]$ such that $\psi(t,\vartheta(t)) = 0$ in $[0,t_0]$ and $\vartheta(0) = \theta_z$.
We also have the derivative
\begin{equation*}
\vartheta'(t) = -\frac{\partial_t \psi(t,\vartheta(t))}{\partial_{\vartheta} \psi(t,\vartheta(t))}
= -\frac{ \zeta(t,\vartheta(t)) \cdot \partial_t  \zeta(t,\vartheta(t))   -  (r+t\dr)\dr}{  \zeta(t,\vartheta(t)) \cdot \partial_{\vartheta}  \zeta(t,\vartheta(t)) },
\end{equation*}
and in particular, using $\nu_i = (\cos\theta_z,  \sin\theta_z)^\top$ and $\tau_i = (-\sin\theta_z,  \cos\theta_z)^\top$, 
\begin{equation}\label{theta_prime_a_2}
\vartheta'(0)
= -\frac{ ( x_i - x_\ell + r \nu_i)\cdot (\dr\nu_i)   - r\dr}{ ( x_i - x_\ell + r \nu_i)\cdot (r \tau_i)}
= \frac{\dr}{r} \left( \frac{1 -   \nu_\ell\cdot \nu_i }{  \nu_\ell\cdot \tau_i} \right).
\end{equation}

We are now ready to build the mapping $T_t$.
Let $\sa(t)\subset \partial B(x_i,r+t\dr)$ be one of the arcs  parameterized by the angle aperture $[\theta_v(t), \theta_w(t) ]$ in the decomposition~\eqref{eq:dec_Eradius}; we have dropped the index $k$ for simplicity. 
Then, $\theta_v(t)$ and $\theta_w(t)$ are given by $\vartheta(t)$ with either  $\theta_z=\theta_v(0)$ or $\theta_z=\theta_w(0)$, and $\vartheta(t)$ either corresponds to a point in $\partial B(x_i,r+t\dr)\cap \partial A$ or to a point in $\partial B(x_i,r+t\dr)\cap\partial B(x_\ell,r+t\dr)$. 

Next, define $\xi(t,\theta)$ and $\alpha(t)$ as in \eqref{114}.
Then, for $\theta\in [\theta_v(0),\theta_w(0)]$ we have
$\xi(t,\theta) \in [\theta_v(t),\theta_w(t)]$ and $\xi(t,\theta)$ is a parameterization of $\sa(t)$.
A point $x\in \sa(0)$ may be parameterized  by
\begin{equation}\label{eq:Tt_r}
x = x_i +r \begin{pmatrix} \cos\theta \\ \sin\theta \end{pmatrix},
\ \text{ and define }\ 
\mathds{T}_t(\theta) :=  x_i + (r+t\dr) \begin{pmatrix} \cos\xi(t,\theta)  \\ \sin\xi(t,\theta) \end{pmatrix} .
 \end{equation}
Writing  $\xi(t,\theta) = \theta + \beta(t,\theta)$ with $\beta(t,\theta) := (\alpha(t) -1)(\theta -\theta_b(t))$, we observe that
$$ \begin{pmatrix} \cos\xi(t,\theta)  \\ \sin\xi(t,\theta) \end{pmatrix}  = R(x_i, \beta(t,\theta))  \begin{pmatrix} \cos\theta \\ \sin\theta \end{pmatrix} = R(x_i, \beta(t,\theta)) \nu_i,$$
where $R(x_i, \beta(t,\theta))$ is a rotation matrix of center $x_i$ and angle $\beta(t,\theta)$.
Also, thanks to $\theta_v(0)<\theta_w(0)<\theta_v(0)+ 2\pi$ and $\theta\in [\theta_v(0),\theta_w(0)]$, there exists a smooth bijection $\theta:\arc\ni x\mapsto \theta(x)\in [\theta_v(0),\theta_w(0)]$.
Thus, using \eqref{eq:Tt_r} we can define the mapping
\begin{equation}\label{eq:Ttarc}
 T_t(x) := \mathds{T}_t(\theta(x))  = x - r \nu_i(x)  +  (r+t\dr)R(x_i, \beta(t,\theta(x))) \nu_i(x)  \text{ for all }x\in  \sa(0)\subset \partial B(x_i,r).  
\end{equation}
The bi-Lipschitz property of  $T_t$ on $\partial\Om(\bx,r)\cap A$ can be obtained as in the proof of \cite[Thm.~3.3]{coveringfirst}.

Finally, differentiating in \eqref{eq:Tt_r} with respect to $t$ and using $\xi(0,\theta)=\theta$ we get \eqref{eq:Vtheta_arc_radius}.
Then \eqref{114} yields  $\xi(t,\theta_a(0)) =  \theta_a(t)$, $\xi(t,\theta_b(0)) =  \theta_b(t)$, $\partial_t\xi(0,\theta_a(0)) = \theta'_a(0)$, $ \partial_t\xi(0,\theta_b(0)) = \theta'_b(0)$, consequently  using \eqref{der_vartheta} we obtain \eqref{eq:Vthetaa_radius}  and using \eqref{theta_prime_a_2} we obtain \eqref{eq:Vthetaa_2_radius}. 
\end{proof}

\subsection{Second-order derivative  of \texorpdfstring{$G$}{G} with respect to the radius}\label{sec:second_D_rG_fast}

The first-order derivative of $G$ with respect to the radius is given by
$$\partial_r G(\bx,r)  =  -\int_{\partial\Om(\bx,r)\cap A}  \, dz,$$
see \eqref{gradG} and \cite[\S3.3]{coveringfirst} for the detailed calculation.
As in \cite{coveringfirst}, the calculation is achieved through integration by substitution using the mapping $T_t$ given by Theorem \ref{thm5}, which requires that Assumption~\ref{a1} and Assumption~\ref{a3} hold.
According to Theorem \ref{thm5}, there exists a bi-Lipschitz mapping $T_t$ satisfying $T_t(\partial\Om(\bx,r)\cap A) = \partial \Om(\bx, r+t\delta r)\cap A$, and this yields, using Lemma \ref{lem:intsub} on each arc of $\partial\Om(\bx,r)\cap A$,
\begin{align*}
\partial_r G(\bx,r + t\dr) &= - \int_{\partial\Om(\bx,r+ t\dr)\cap A}\, dz
= - \int_{T_t(\partial\Om(\bx,r)\cap A)} \, dz
= - \int_{\partial\Om(\bx,r)\cap A} \omega_t(z) \, dz.
\end{align*}
Thus, using Lemma \ref{lem:intsub} and the decomposition \eqref{unionAi}, we compute
\begin{align*}
\left.\frac{d}{dt} \partial_r G(\bx,r + t\dr)\right|_{t= 0} 
& = - \int_{\partial\Om(\bx,r)\cap A} \divv_\Gamma V(z) \, dz
 = - \sum_{i=1}^m \int_{\arc_i} \divv_\Gamma V(z) \, dz.
\end{align*}
Applying Theorem \ref{lem:tandivpol} for each arc in $\arc_i$, we obtain
\begin{align}\label{eq:333}
\left.\frac{d}{dt} \partial_r G(\bx,r + t\dr)\right|_{t= 0} 
& 
 = - \sum_{i=1}^m \int_{\arc_i} \mathcal{H} V\cdot \nu_i \, dz
 - \sum_{i=1}^m  \sum_{(v,w) \in \pairsvw_i} \llbracket V(z)\cdot \tau_i(z)\rrbracket_v^w.
 \end{align}
 
To get a more explicit formula we need to determine $V(v),V(w)$ and $V\cdot \nu_i$ on $\arc_i$.
For this we apply Theorem \ref{thm5} to two different cases.
On the one hand, if $v\in\partial B(x_i,r)\cap \partial B(x_\ell,r)$ for some  $i\neq\ell$, then applying \eqref{eq:Vthetaa_2_radius} we obtain
\begin{align}\label{zkin}
 V(v)\cdot \tau_i(v) =  \dr  \frac{1 - \nu_\ell(v)\cdot \nu_i(v)}{\nu_\ell(v)\cdot \tau_i(v)}. 
\end{align}
On the other hand, if $v\in\partial B(x_i,r)\cap \partial A$, then applying 
\eqref{eq:Vthetaa_radius}
we get
\begin{equation}\label{zkin2}
V(v)\cdot \tau_i(v) = - \dr\frac{\nu_A(v)\cdot \nu_i(v)}{\tau_i(v) \cdot \nu_A(v)}.
\end{equation}
Then, recalling that $L(z)=\{ \ell \in \{1,\dots,m\} \setminus \{i\} \;|\; z \in \partial B(x_\ell,r) \}$ for $z \in \{v,w\}$, and that $\nu_{-i}(z):=\nu_\ell(z)$ if $z\in\partial B(x_i,r)\cap \partial B(x_\ell,r)$, $\ell\neq i$, and $\nu_{-i}(z):=\nu_A(z)$ if $z\in\partial B(x_i,r)\cap \partial A$,
 we can merge \eqref{zkin} and \eqref{zkin2} into a unique formula:
\begin{equation}\label{zkin3}
V(v)\cdot \tau_i(v) = \dr\frac{|L(v)| - \nu_{-i}(v)\cdot \nu_i(v)}{\nu_{-i}(v)\cdot\tau_i(v)}.
\end{equation}
In a similar way, we also obtain
\begin{equation*}
V(w)\cdot \tau_i(w) = \dr\frac{|L(w)| - \nu_{-i}(w)\cdot \nu_i(w)}{\nu_{-i(w)}\cdot\tau_i(w)}.
\end{equation*}

Gathering these results we get
\begin{align*}
\left.\frac{d}{dt} \partial_r G(\bx,r + t\dr)\right|_{t= 0} 
& 
 = -  \dr\frac{\mathrm{Per}(\partial\Om(\bx,r)\cap A)}{r}
 - \dr  \sum_{i=1}^m  \sum_{(v,w) \in \pairsvw_i} \left\llbracket \, \frac{|L(z)| - \nu_{-i}(z) \cdot \nu_i(z)}{\nu_{-i}(z) \cdot \tau_i(z)} \, \right\rrbracket_v^w,
 \end{align*}
where we have used 
$\mathcal{H} V\cdot \nu_i = \frac{\dr}{r} \text{ on } \arc_i\subset\partial B(x_i,r)$ due to  \eqref{eq:Vtheta_arc_radius} and $\mathcal{H}=1/r$.
Thus we have obtained \eqref{d2Grr}.

\subsection{Second-order derivative of \texorpdfstring{$G$}{G} with respect to the centers}\label{sec:second_D_xiG}

The first-order derivative of $G$ with respect to the center $x_i$ is given by
$$\partial_{x_i} G(\bx,r) = - \int_{\arc_i} \nu_i(z) \, dz.$$
see \eqref{gradG} and \cite[\S~3.4]{coveringfirst} for the detailed calculation.
As in \cite{coveringfirst}, the calculation is achieved through integration by substitution using the mapping $T_t$ from Theorem \ref{thm:EB} with the specific  perturbation $\delta \bx = (0,\dots,0,\delta x_i, 0,\dots, 0)$, which requires that Assumption~\ref{a1} and Assumption~\ref{a3} hold.
Using Lemma \ref{lem:intsub} yields 
\begin{align*}
\partial_{x_i} G(\bx + t\delta \bx,r) 
&= - \int_{\partial B(x_i+ t\delta x_i,r)\cap \partial\Om(\bx+ t\delta \bx,r)\cap A} \nu_t(z) \, dz
= - \int_{T_t(\arc_i)} \nu_t(z) \, dz\\
& = - \int_{\arc_i} \nu_t(T_t(z)) \omega_t(z)  \, dz,
\end{align*}
where $\nu_t$ is the outward unit normal vector to $\partial B(x_i+ t\delta x_i,r)\cap \partial\Om(\bx+ t\delta \bx,r)\cap A$ and $\omega_t$ is given by \eqref{density:changevar}.

To obtain the derivative of $\partial_{x_i} G(\bx + t\delta \bx,r) $ with respect to $t$ at $t=0$ we need  the so-called {\it material derivative} of the normal vector given by
$$\frac{d}{dt}\nu_t(T_t(z))|_{t=0} = - (D_\Gamma V)^\top \nu_i \text{ on } \arc_i,$$
with $V := \partial_t T_t|_{t= 0}$; see [Walker, Lemma 5.5, page 99].
Here, $D_\Gamma V: = DV - (DV)\nu_i\otimes\nu_i$ denotes the tangential Jacobian of $V$ on $\arc_i$.
Then, using \eqref{density:der} we obtain
\begin{align*}
\left.\frac{d}{dt}  \partial_{x_i} G(\bx + t\delta \bx,r) \right|_{t= 0} 
& = - \int_{\arc_i} - (D_\Gamma V)^\top \nu_i + \nu_i\divv_\Gamma(V)  \, dz.
\end{align*}
This expression can be further transformed  using the following tensor relations:
\begin{align}
\label{tensor:rel}\divv_\Gamma(\nu_i\otimes V) = \divv_\Gamma (V)\nu_i + (D_\Gamma \nu_i)V
\text{ and }  \nabla_\Gamma (V\cdot\nu_i) = D_\Gamma \nu_i^\top V + D_\Gamma V^\top \nu_i \text{ on } \arc_i.
\end{align}
We show that $D_\Gamma \nu_i^\top V = D_\Gamma \nu_i V$ on $\arc_i$. 
Indeed, let $W\in\R^2$ and denote $V_\tau$ and $W_\tau$ the tangential components of $V$ and $W$ on $\arc_i$.
Differentiating $\nu_i\cdot\nu_i = 1$  on $\arc_i$ we get $(D_\Gamma \nu_i)^\top \nu_i = 0$ and then
\begin{align*}
(D_\Gamma \nu_i)^\top V \cdot W 
= (D_\Gamma \nu_i)^\top V_\tau \cdot W
= (D_\Gamma \nu_i)^\top V_\tau \cdot W_\tau
= (D_\Gamma \nu_i) V_\tau \cdot W_\tau,
\end{align*}
where we have used the  well-known fact that the second fundamental form $(V_\tau,W_\tau)\mapsto (D_\Gamma \nu_i) V_\tau \cdot W_\tau$ is symmetric. Further,
\begin{align}\label{sym:dgamma_nu}
(D_\Gamma \nu_i)^\top V \cdot W 
= (D_\Gamma \nu_i)^\top W_\tau \cdot V_\tau
= (D_\Gamma \nu_i)^\top W \cdot V
= (D_\Gamma \nu_i) V \cdot W  \text{ on } \arc_i.
\end{align}
Now, using \eqref{tensor:rel}, \eqref{sym:dgamma_nu} we obtain
\begin{align*}
\left.\frac{d}{dt}  \partial_{x_i} G(\bx + t\delta \bx,r) \right|_{t= 0} & = - \int_{\arc_i} \divv_\Gamma(\nu_i\otimes V)  - \nabla_\Gamma (V\cdot\nu_i)\, dz.
\end{align*}
Applying Theorem \ref{lem:tandivpol} to the integral of  $\divv_\Gamma(\nu_i\otimes V)$ on each arc in $\arc_i$ we get
\begin{align*}
\left.\frac{d}{dt}  \partial_{x_i} G(\bx + t\delta \bx,r) \right|_{t= 0} 
& = - \int_{\arc_i} \mathcal{H}(\nu_i\otimes V)\cdot \nu_i - \nabla_\Gamma (V\cdot\nu_i) \, dz
- \sum_{(v,w) \in \pairsvw_i}   \llbracket(\nu_i(z)\otimes V(z))\cdot \tau_i(z)\rrbracket_v^w,
 \end{align*}
and then, using $\mathcal{H}=1/r$ on $\arc_i$,
 \begin{align}\label{eq:577}
 \begin{split}
\left.\frac{d}{dt}  \partial_{x_i} G(\bx + t\delta \bx,r) \right|_{t= 0}  
& = - \frac{1}{r}  \int_{\arc_i}(V\cdot \nu_i)\nu_i    -    r\nabla_\Gamma (V\cdot\nu_i) \, dz
- \sum_{(v,w) \in \pairsvw_i} \llbracket  (V(z)\cdot \tau_i(z))\nu_i(z)\rrbracket_v^w.
\end{split}
\end{align}

Applying \eqref{eq:Vtheta_arc} yields
$V\cdot \nu_i  =  \dx_i\cdot \nu_i \text{ on } \arc_i.$
Considering that $\dx_\ell=0$ for $\ell\neq i$ since we use the specific perturbation $\delta \bx = (0,\dots,0,\delta x_i, 0,\dots, 0)$, \eqref{eq:Vthetaa} and \eqref{eq:Vthetaa_2}
actually provide the same formula in this particular case:
\begin{align}\label{eq:vtau}
\begin{split}
V(z)\cdot \tau_i(z) &=  \dx_i\cdot \left(\tau_i  - \frac{\nu_{-i}}{\tau_i \cdot \nu_{-i}}\right)(z)\\
& = - \left(\frac{\nu_{-i} \cdot \nu_i}{ \nu_{-i}\cdot \tau_i }   \dx_i\cdot \nu_i\right)(z) \text{ for }z\in \{v, w\} \text{ and } (v,w)\in\pairsvw_i.
\end{split}
\end{align}
We also have
$\nabla_\Gamma (V\cdot\nu_i) 
=   \nabla_\Gamma (\dx_i\cdot \nu)
= (D_\Gamma\nu)^\top  \dx_i$ and
$$ D_\Gamma\nu_i = D_\Gamma \begin{pmatrix} \cos\theta \\ \sin\theta\end{pmatrix}
= \begin{pmatrix} \nabla_\Gamma (\cos\theta)^\top \\ \nabla_\Gamma (\sin\theta)^\top\end{pmatrix}
= 
\frac{1}{r}\begin{pmatrix} \partial_\theta (\cos\theta)\tau_i^\top \\ \partial_\theta (\sin\theta)\tau_i^\top\end{pmatrix}
=  \frac{1}{r}  \tau_i\otimes \tau_i
\text{ on }\arc_i.$$
Gathering these results and using $V\cdot \nu_i  =  \dx_i\cdot \nu_i \text{ on } \arc_i$ we get
\begin{align*}
\left.\frac{d}{dt}  \partial_{x_i} G(\bx + t\delta \bx,r) \right|_{t= 0} 
& = - \frac{1}{r}  \int_{\arc_i}(\dx_i\cdot \nu_i)\nu_i  - (\tau_i\otimes \tau_i)\dx_i\, dz
+  \sum_{(v,w) \in \pairsvw_i}   \left\llbracket\frac{\nu_{-i} \cdot \nu_i}{ \nu_{-i}\cdot \tau_i }  \nu_i\otimes  \nu_i\right\rrbracket_v^w \dx_i,
\end{align*}
which yields \eqref{d2xixiG}.

\subsection{Second order derivative with respect to \texorpdfstring{$x_i$}{xi} and \texorpdfstring{$x_\ell$}{xl} of \texorpdfstring{$G$}{G}}\label{sec:derxixl}

As in Section \ref{sec:second_D_xiG} we use the mapping $T_t$ from Theorem \ref{thm:EB}, which requires that Assumptions~\ref{a1} and \ref{a3} hold, but now  with the specific  perturbation $\delta \bx = (0,\dots,0,\delta x_\ell, 0,\dots, 0)$.
This yields a transformation $T_t$ satisfying in particular $T_t(\arc_i) = \partial B(x_i,r)\cap \partial\Om(\bx+ t\delta \bx,r)\cap A$.
Then, using Lemma \ref{lem:intsub} we obtain
\begin{align*}
\partial_{x_i} G(\bx + t\delta \bx,r) 
&= - \int_{\partial B(x_i,r)\cap \partial\Om(\bx+ t\delta \bx,r)\cap A} \nu_t(z) \, dz
= - \int_{T_t(\arc_i)} \nu_t(z) \, dz
= - \int_{\arc_i} \nu_t(T_t(z)) \omega_t(z)  \, dz,
\end{align*}
where $\nu_t$ is the outward unit normal vector to $\partial B(x_i,r)\cap \partial\Om(\bx+ t\delta \bx,r)\cap A$ and $\omega_t$ is given by \eqref{density:changevar}.
Applying \eqref{eq:Vtheta_arc} and considering that $\dx_i=0$ since we are using the specific  perturbation $\delta \bx = (0,\dots,0,\delta x_\ell, 0,\dots, 0)$, we get
\begin{equation}\label{eq:743}
V\cdot \nu_i  =   0
\text{ on } \arc_i.
\end{equation}
Then, applying \eqref{eq:Vthetaa_2} with $\dx_i = 0$ we get
\begin{align}\label{eq:744}
 V(z)\cdot\tau_i(z) 
 &= \frac{\dx_\ell\cdot \nu_\ell(z)}{\tau_i(z)\cdot \nu_\ell(z)}\quad  \text{ if } z\in \partial B(x_\ell,r)\cap\partial B(x_i,r), i\neq \ell.
\end{align}

Next, the derivative of $\partial_{x_i} G(\bx + t\delta \bx,r)$ with respect to $t$ at $t=0$ is already calculated in \eqref{eq:577}, but the terms $(V\cdot \nu_i)\nu_i$ and  $\nabla_\Gamma (V\cdot\nu_i)$ in \eqref{eq:577} vanish due to \eqref{eq:743}.
We also observe that $V(z) = 0$ if $z\in\{v,w\}$ with $(v,w) \in \pairsvw_i$ and $z\notin \partial B(x_\ell,r)$.
Finally, using ${\cal I}_{i\ell} = \{ v \in \partial B(x_\ell,r) \; | \; (v,\cdot) \in \pairsvw_i \}$, $\mathcal{O}_{i\ell} = \{ w \in \partial B(x_\ell,r) \; | \; (\cdot,w) \in \pairsvw_i \}$
 and \eqref{eq:744} we get
\begin{align*}
\left.\frac{d}{dt}  \partial_{x_i} G(\bx + t\delta \bx,r) \right|_{t= 0} 
& = \sum_{v \in {\cal I}_{i\ell}}   V(v)\cdot \tau_i(v)\nu_i(v)  - \sum_{w \in {\cal O}_{i\ell}}    V(w)\cdot \tau_i(w)\nu_i(w) \\
 & = \left[   \sum_{v \in {\cal I}_{i\ell}}   \frac{\nu_i(v)\otimes \nu_\ell(v)}{\nu_\ell(v)\cdot \tau_i(v)} 
 -\sum_{w \in {\cal O}_{i\ell}}   \frac{\nu_i(w)\otimes \nu_\ell(w)}{\nu_\ell(w)\cdot \tau_i(w)}\right] \dx_\ell,
\end{align*}
which yields \eqref{d2Gxixj}.

\subsection{Second order derivative with respect to \texorpdfstring{$x_i$}{xi} and \texorpdfstring{$r$}{r} of \texorpdfstring{$G$}{G}}\label{sec:derxir}

In a similar way as in Sections \ref{sec:second_D_xiG} and \ref{sec:derxixl},
we use the mapping $T_t$ from Theorem \ref{thm:EB} with the specific  perturbation $\delta \bx = (0,\dots,0,\delta x_i, 0,\dots, 0)$. 
This yields, using Lemma \ref{lem:intsub},
\begin{align*}
\partial_r G(\bx + t\delta \bx,r) &= - \int_{\partial\Om(\bx + t\delta \bx,r)\cap A}\, dz
= - \int_{T_t(\partial\Om(\bx,r)\cap A)} \, dz
= - \int_{\partial\Om(\bx,r)\cap A} \omega_t(z) \, dz.
\end{align*}
Proceeding as in the calculation leading to \eqref{eq:333}, we get
\begin{align}\label{eq:777}
\begin{split}
\left.\frac{d}{dt} \partial_r G(\bx + t\delta \bx,r)\right|_{t= 0} 
& 
 = - \sum_{\ell=1}^m   \int_{\arc_\ell} \mathcal{H} V\cdot \nu_\ell \, dz
 - \sum_{\ell=1}^m  \sum_{(v,w) \in \pairsvw_\ell} \llbracket V(z)\cdot \tau_\ell(z)\rrbracket_v^w.
\end{split}
\end{align}
 
Considering that $\dx_\ell=0$ for $\ell\neq i$, since we use the specific perturbation $\delta \bx = (0,\dots,0,\delta x_i, 0,\dots, 0)$, \eqref{eq:Vthetaa} and \eqref{eq:Vthetaa_2}
actually provide the same formula in this particular case:
\begin{align}\label{eq:772b}
\begin{split}
V(z)\cdot \tau_i(z) &=  \dx_i\cdot \left(\tau_i  - \frac{\nu_{-i}}{\tau_i \cdot \nu_{-i}}\right)(z)\\
& = - \left(\frac{\nu_{-i} \cdot \nu_i}{ \nu_{-i}\cdot \tau_i }   \dx_i\cdot \nu_i\right)(z) \text{ for }z\in \{v, w\} \text{ and } (v,w)\in\pairsvw_i, 
\end{split}
\end{align}
and also 
\begin{align*} 
V(z)\cdot \tau_\ell(z)
&=  \dx_i \cdot \left( \tau_\ell - \frac{\nu_\ell}{\tau_i \cdot \nu_\ell} (\tau_i\cdot\tau_\ell)\right)(z)\\
&= \dx_i \cdot \left( \frac{\mu}{\tau_i \cdot \nu_\ell} \right)(z)
 \text{ if } z\in\partial B(x_i,r)\cap \partial B(x_\ell,r) \text{ and } \ell\neq i.  
\end{align*}
with $\mu:= \tau_\ell (\tau_i\cdot \nu_\ell) - (\tau_i\cdot \tau_\ell)\nu_\ell$. This yields $\mu\cdot\tau_i = 0$ and 
$$ \mu\cdot \nu_i = (\tau_\ell\cdot \nu_i)(\tau_i\cdot \nu_\ell) - (\tau_i\cdot \tau_\ell)(\nu_\ell\cdot \nu_i)
= - (\tau_i\cdot \nu_\ell)^2 - (\tau_i\cdot \tau_\ell)^2 = -1,$$
where we have used the geometric properties $\tau_\ell\cdot \nu_i = - \tau_i\cdot \nu_\ell$ and $\tau_i\cdot \tau_\ell = \nu_\ell\cdot \nu_i$.
Thus  $\mu = -\nu_i$ and  we get
\begin{align}\label{eq:774}
 V(z)\cdot \tau_\ell(z)
= - \frac{\dx_i \cdot \nu_i(z)}{\tau_i(z) \cdot \nu_\ell(z)}  \text{ if } z\in\partial B(x_i,r)\cap \partial B(x_\ell,r) \text{ and } \ell\neq i. 
\end{align}

In \eqref{eq:777}, we observe that $V(z)=0$ whenever $z\in\{v,w\}$ and $z\notin \partial B(x_i,r)$; this can be seen from \eqref{eq:Vthetaa}-\eqref{eq:Vthetaa_2} and the fact that we use the specific  perturbation $\delta \bx = (0,\dots,0,\delta x_i, 0,\dots, 0)$. 
Hence, recalling that  $L(z)=\{ \ell \in \{1,\dots,m\} \setminus \{i\} \;|\; z \in \partial B(x_\ell,r) \}$,
\begin{align*}
\sum_{\ell=1}^m  \sum_{(v,w) \in \pairsvw_\ell} \llbracket V(z)\cdot \tau_\ell(z)\rrbracket_v^w
&= 
\sum_{(v,w) \in \pairsvw_i} \llbracket V(z)\cdot \tau_i(z)\rrbracket_v^w 
+
\sum_{\substack{\ell=1 \\ \ell\neq i}}^m  \sum_{(v,w) \in \pairsvw_\ell} \llbracket V(z)\cdot \tau_\ell(z)\rrbracket_v^w\\
&= 
\sum_{(v,w) \in \pairsvw_i} \llbracket V(z)\cdot \tau_i(z)\rrbracket_v^w 
-
\sum_{(v,w) \in \pairsvw_i}   \left\llbracket \sum_{\ell\in L(z)}   V(z)\cdot \tau_\ell(z)\right\rrbracket_v^w.
\end{align*}
Note that the negative sign  in front of the last sum is due to the fact that if an ending point of an arc in $\pairsvw_\ell$ belongs to some arc in $\pairsvw_i$, then it is a starting point for this arc in $\pairsvw_i$, and vice versa.
Using \eqref{eq:Vtheta_arc} we have $V\cdot \nu_\ell \equiv 0$ on $\arc_\ell$ for all $\ell\neq i$. 
Since $\mathcal{H} = 1/r$, we may write \eqref{eq:777} as
\begin{align}\label{eq:778}
\begin{split}
\left.\frac{d}{dt} \partial_r G(\bx + t\delta \bx,r)\right|_{t= 0} 
& 
 = - \frac{1}{r}\int_{\arc_i} V\cdot \nu_i \, dz
 -  \sum_{(v,w) \in \pairsvw_i} \left\llbracket V(z)\cdot \tau_i(z)
-
\sum_{\ell\in L(z)}   V(z)\cdot \tau_\ell(z)\right\rrbracket_v^w.
\end{split}
\end{align}
Using \eqref{eq:Vtheta_arc} we get $V\cdot \nu_i = \dx_i\cdot\nu_i$ on $\arc_i$.
Finally, using \eqref{eq:772b}-\eqref{eq:774} we get
\begin{align*}
\begin{split}
\left.\frac{d}{dt} \partial_r G(\bx + t\delta \bx,r)\right|_{t= 0} 
& 
 = -  \frac{1}{r} \int_{\arc_i} \dx_i\cdot\nu_i \, dz\\
&\quad +  \sum_{(v,w) \in \pairsvw_i} 
\left\llbracket\frac{\nu_{-i}(z) \cdot \nu_i(z)}{ \nu_{-i}(z)\cdot \tau_i(z) }   \dx_i\cdot \nu_i(z)
 -  \sum_{\ell\in L(z)}   \frac{\dx_i \cdot \nu_i(z)}{\tau_i(z) \cdot \nu_\ell(z)} 
 \right\rrbracket_v^w ,
\end{split}
\end{align*}
which yields \eqref{eq:781}.

\section{Analysis of singular cases}\label{sec:singularcases}

The gradient $\nabla G$ and Hessian $\nabla^2 G$ were obtained under Assumptions~\ref{a1} and \ref{a3}, and in this section we investigate several singular cases where these assumptions are not satisfied.
On the one hand, it is shown in \cite[\S~3.5]{coveringfirst} that $G$ is often differentiable even when Assumptions~\ref{a1} and \ref{a3} do not hold, and in the few cases where  $G$ is not differentiable it is at least Gateaux semidifferentiable.
On the other hand,  $G$ is never twice differentiable in any of the singular geometric configurations studied in this section. 
Nevertheless,  Gateaux semidifferentiability of the components of $\nabla G$ can often be proven.

We recall here that $f:\R^n\to\R$ is Gateaux semidifferentiable at $x$ in the direction $v$ if
$$ \lim_{t\searrow 0} \frac{f(x+tv) - f(x) }{t} \text{ exists in } \R^n,$$
and that $f$ has a derivative in the direction $v$ at $x$ if
$$ \lim_{t\to 0} \frac{f(x+tv) - f(x) }{t} \text{ exists in } \R^n.$$

\begin{example} \label{examp1}
Suppose $m=2$, $\Om(\bx +t\delta\bx,r)\subset A$ for all $t\in [0,t_0]$ and $t_0$ sufficiently small, and the two balls are tangent at $t=0$, i.e., $\| x_1 -x_2 \|= 2r$; thus Assumption \ref{a1} is not satisfied.
Two cases need to be considered to compute the gradient of $G$.
First, if $(x_1 -x_2)\cdot (\dx_1 - \dx_2) \geq  0$, then it is clear that $B(x_1 +t\dx_1,r)\cap B(x_2+t\dx_2,r) =\emptyset$ for all $t\in [0,t_0]$. 
Therefore $G(\bx+t\delta\bx,r)=G(\bx,r)=\vol(A) - 2\pi r^2$ for all $t\in [0,t_0]$, and  $ \lim_{t\searrow 0} (G(\bx+t\delta\bx,r) - G(\bx,r) )/t = 0$.
Second, if $(x_1 -x_2)\cdot (\dx_1 - \dx_2) <  0$ then $B(x_1 +t\dx_1,r)\cap B(x_2+t\dx_2,r) \neq \emptyset$ for all $t \in (0,t_0]$.
The intersection of $B(x_1 +t\dx_1,r)$ and $B(x_2+t\dx_2,r)$ form a symmetric lens whose area is given by 
$$a(t) = 2r^2 \arccos \left( d(t)/ 2r \right) -d(t) \left( r^2 - d(t)^2/4 \right)^{1/2},$$
where $d(t):= \|x_1 + t\dx_1 - (x_2 +t\dx_2)\|$.
It is convenient to rewrite this expression as
$$ a(t) = 2r^2 \arccos \left( (1-g(t))^{1/2} \right) - 2r^2\left( g(t) + g(t)^2 \right)^{1/2},$$
with $g(t):= -(2t  (x_1 -x_2)\cdot (\dx_1 - \dx_2) + t^2 \| \dx_1 - \dx_2 \|^2)/(4r^2)$, $g(t)\geq 0$ for all $t\in [0,t_0]$ for $t_0$ small enough, 
$ d(t) =2r ( 1- g(t))^{1/2}$, and $g'(0) = -  (x_1 -x_2)\cdot (\dx_1 - \dx_2)/(2r^2)$.
After simplifications, we obtain
$ a'(t) = 2r^2 \left(  \frac{g(t)}{1 -g(t)} \right)^{1/2}  g'(t),$
and in particular $a'(0) = 0$.
This shows that
$$ \lim_{t\searrow 0} \frac{G(\bx+t\delta\bx,r) - G(\bx,r) }{t} = 0 \quad \text{  when } (x_1 -x_2)\cdot (\dx_1 - \dx_2) <  0.$$
Hence $\lim_{t\to 0} (G(\bx+t\delta\bx,r) - G(\bx,r) )/t = 0$  for all $\delta\bx\in\R^4$.
Proceeding in a similar way we can also show that
$ \lim_{t\to 0} (G(\bx,r +t) - G(\bx,r) )/t = -4\pi r$.
Thus $\nabla G(\bx,r) = (0,\dots,0,-4\pi r)^\top$ in the case $\| x_1 -x_2 \|= 2r$. (In \cite[Example~3.10]{coveringfirst} it is written $\nabla G(\bx,r) = (0,\dots,0,4\pi r)^\top$ where it should be written $\nabla G(\bx,r) = (0,\dots,0,-4\pi r)^\top$.)

It is easy to check that formula \eqref{gradG} also gives $\nabla G(\bx,r) =  (0,\dots,0,-4\pi r)^\top$ in this case.
This indicates that, for the analyzed case, \eqref{gradG} is valid even without the satisfaction of Assumption~\ref{a1}. 
However, we had to use a different technique to prove that \eqref{gradG} holds, as $G(\bx+t\delta\bx,r)$ takes different expressions depending on the sign of  $(x_1 -x_2)\cdot (\dx_1 - \dx_2)$.

We now study second-order differentiability of $G$.
Let $f(t):=G(\bx+t\delta\bx,r)$, then  if $(x_1 -x_2)\cdot (\dx_1 - \dx_2) \geq  0$ we have
$$f(t)=\vol(A) - \vol(B(x_1+t\dx_1,r)) - \vol(B(x_2+t\dx_2,r))
= \vol(A) -2\pi r^2 \text{ for all } t\in [0,t_0].$$
Thus in this case we get the right derivatives $f'(0)=\nabla_{\bx} G(\bx,r)\cdot\delta\bx=0$ and $f''(0)=\nabla^2_{\bx} G(\bx,r)\delta\bx\cdot\delta\bx = 0$.

In the case $(x_1 -x_2)\cdot (\dx_1 - \dx_2) \leq  0$ we get
$$f(t)=\vol(A) - \vol(B(x_1+t\dx_1,r)) - \vol(B(x_2+t\dx_2,r)) +a(t) \text{ for all } t\in [0,t_0],$$
and $f'(t)=\nabla_{\bx} G(\bx+t\delta\bx,r)\cdot\delta\bx=a'(t)$,  $f''(t)=\nabla^2_{\bx} G(\bx+t\delta\bx,r)\delta\bx\cdot\delta\bx = a''(t)$.
The calculation yields
$$ a''(t) = r^2 \left(  \frac{g(t)}{1 -g(t)} \right)^{-1/2}\left( \frac{g'(t)}{1 -g(t)} + \frac{g(t)g'(t)}{(1 -g(t))^2}\right)  g'(t)
+ 2r^2 \left(\frac{g(t)}{1 -g(t)} \right)^{1/2} g''(t),$$
and $g(0)=0$, $g'(0) = -  (x_1 -x_2)\cdot (\dx_1 - \dx_2)/(2r^2)$, $g''(0)= - \| \dx_1 - \dx_2 \|^2/(2r^2)$.
This shows that $|a''(t)|\to\infty$ as $t\to 0$ and consequently  $|\nabla^2_{\bx} G(\bx+t\delta\bx,r)\delta\bx\cdot\delta\bx| \to\infty$  as $t\to 0$.
This result is coherent with \eqref{d2xixiG} as we can show that $|\partial^2_{x_1x_1} G(\bx,r)u\cdot u| \to\infty$  as $\|x_1 - x_2\|\to 2r$ with $\|x_1 - x_2\|< 2r$ and $u=(1,0)^\top$.
Thus $\nabla_{\bx} G$ is not  differentiable in this geometric configuration.

Now we investigate the Gateaux semidifferentiability of $\partial_r G$ with respect to the radius.
Let us introduce the notation 
$L(\rho) := \vol \left(B(x_1,\rho)\cap B(x_2,\rho)\right)$
with $\rho>r>\|x_1-x_2\|/2.$
Then $L(\rho)$ is the area of a symmetric lens given by
$$ L(\rho) = 2\rho^2 \arccos \left( \frac{r}{\rho} \right) - 2r\left( \rho^2 - r^2 \right)^{1/2},$$
and we have
$G(\bx,\rho)=\vol(A) - 2\pi \rho^2 +L(\rho) \text{ for } 2r>\rho\geq r$.
Thus
\begin{align*}
\partial_\rho G(\bx,\rho) &= - 4\pi \rho +L'(\rho)\text{ and }
\partial^2_{\rho\rho} G(\bx,\rho) = - 4\pi +L''(\rho) \text{ for } 2r>\rho> r,
\end{align*}
and we compute
\begin{align*}
L'(\rho) &= 4\rho \arccos \left( \frac{r}{\rho} \right),\qquad
L''(\rho) = 4 \arccos \left( \frac{r}{\rho} \right) + \frac{4r}{(\rho^2 -r^2)^{1/2}} \text{ for } \rho>r.
\end{align*}
Then we observe that
$L'(r) = 0$ and $\lim_{\rho\searrow r} L''(\rho) = +\infty$.
Thus,  if $\dr>0$ then 
$$ \lim_{t\searrow 0} \frac{\partial_r G(\bx,r+t\dr) - \partial_r G(\bx,r) }{t} = +\infty,$$
and $\partial_r G$ is not Gateaux semidifferentiable in direction $(0,0,\dr)$ with $\dr>0$.

On the other hand, if $\dr<0$ and $t>0$, then we have $G(\bx,r+t\dr)=\vol(A) - 2\pi (r+t\dr)^2$, thus $\partial_r G$ is Gateaux semidifferentiable in direction $(0,0,\dr)$ and 
$$ \lim_{t\searrow 0} \frac{\partial_r G(\bx,r+t\dr) - \partial_r G(\bx,r) }{t} = -4\pi\dr \text{ for }\dr<0.$$
We conclude that $\partial_r G$ is not differentiable in this geometric configuration.
\end{example}


\begin{example} \label{examp2}
Suppose $A$ is a square, $m=1$, $\Om(\bx,r)\subset A$ and $\Om(\bx,r)$ is tangent to $\partial A$ on the right side of the square but is not tangent to the other sides. Note that Assumption \ref{a3} is not satisfied as $\Om(\bx,r)$ and $A$ are not compatible.
Then $\delta\bx = \dx_1$ and for sufficiently small $t>0$ we have
\begin{align*}
G(\bx+t\delta\bx,r) &=\vol(A) - \pi r^2 \text{ if } \dx_1 = (-1,0)^\top\\
G(\bx+t\delta\bx,r) &=\vol(A) - \pi r^2 + a(t) \text{ if } \dx_1 = (1,0)^\top.
\end{align*}
with 
$$a(t) = r^2\arccos\left(\frac{r-t\dx_1}{r}\right) -(r-t\dx_1)g(t)^{1/2} $$
and $g(t)=(r^2 -(r-t\dx_1)^2)$. We compute
$a'(t) = 2\dx_1 g(t)^{1/2}$, $a''(t) = \dx_1 g(t)^{-1/2}g'(t)$, $g(0)=0$ and $g'(0)=2\dx_1 r$.
Thus $a'(0)=0$ and $a''(t)\to +\infty$ as $t\to 0$.
It is clear that $G(\bx+t\delta\bx,r)$ does not depend on the second component of $\dx_1$ for sufficiently small $t>0$, thus we have shown that $G$ has a derivative in direction $(\delta\bx,0)$ for all $\delta\bx\in\R^2$ and that $\partial_{x_1} G(\bx,r)=0$, which gives the same value as \eqref{gradG} even though $\Om(\bx,r)$ and $A$ are not compatible in this example.

Since $G(\bx+t\delta\bx,r)$ is constant for $\delta\bx = (-1,0)^\top$, we have
$$ \lim_{t\searrow 0} \frac{\partial_{x_1} G(\bx+t\delta\bx,r) - \partial_{x_1} G(\bx,r) }{t} = (0,0)^\top \quad \text{  for } \delta\bx = (-1,0)^\top,$$
thus $\partial_{x_1} G$ is Gateaux semidifferentiable in direction $(\delta\bx,0)$ with $\delta\bx=\dx_1 = (-1,0)^\top$.
On the other hand for $\delta\bx = (1,0)^\top$ we have
$$ \lim_{t\searrow 0} \frac{( \partial_{x_1} G(\bx+t\delta\bx,r) - \partial_{x_1} G(\bx,r) ) \cdot\delta\bx}{t} =  \lim_{t\searrow 0}  a''(t) = +\infty \quad \text{  for } \delta\bx = (1,0)^\top,$$
thus $\partial_{x_1} G$ is not Gateaux semidifferentiable in direction $(\delta\bx,0)$ with $\delta\bx=\dx_1 = (1,0)^\top$.
This shows that $G$ is not twice differentiable in this geometric configuration.
\end{example}


\begin{example} \label{examp2b}
Let $A = [0,2]^2$, $m=1$, $\Om(\bx,r)=B(x_1,r)$ with $x_1 = (0,1/2)$ and $r=1/2$, then $\partial B(x_1,r)$ intersects $\partial A$ at a vertex, thus $A$ and $\Om(\bx,r)$ are not compatible and Assumption~\ref{a3} is not satisfied. 

In the case of a horizontal translation $\dx_1 = (1,0)^\top$ we symmetrize the square $A$ vertically by defining $A_s = A\cup [0,2]\times [-2,0]$. 
Defining $G_s(\bx,r) =  \vol(A_s) - \vol(A_s \cap \Om(\bx,r))$
we observe that $G_s(\bx+t\delta\bx,r) = G(\bx+t\delta\bx,r) + \vol([0,2]\times [-2,0])$ for $\delta\bx = \dx_1 = (1,0)^\top$ and sufficiently small $t$, so that 
$G_s(\bx,r)$  and $G(\bx,r)$ have the same partial derivatives in direction $\dx_1 = (1,0)^\top$ since $\vol([0,2]\times [-2,0])$ is constant.
Since $A_s$ and $\Om(\bx,r)$ are compatible, this shows that $\partial_{x_1}  G(\bx,r)$ has a derivative  at $(\bx,r)$  in direction $(\delta\bx,0)$ with  $\delta\bx=\dx_1 = (1,0)^\top$.

In the case of a vertical translation $\dx_1 = (0,\pm 1)^\top$ we symmetrize the square $A$ horizontally by defining $A_s = A\cup [-2,0]\times [0,2]$.
Then $\Om(\bx,r)$ is tangent to one side of $A_s$ and we can use the results of  Example~\ref{examp2}.
We conclude that $\partial_{x_1} G$ is Gateaux semidifferentiable in direction $(\delta\bx,0)$ with $\delta\bx=\dx_1 = (0,1)^\top$ but is not Gateaux semidifferentiable in direction $(\delta\bx,0)$ with $\delta\bx=\dx_1 = (0,-1)^\top$.
This shows that $G$ is not twice differentiable in this geometric configuration.
\end{example}

\begin{example} \label{examp3}
Let $m=3$ and $x_1,x_2,x_3$ be the vertices of an equilateral triangle. The circles $\partial B(x_1,r)$, $\partial B(x_2,r)$ and  $\partial B(x_3,r)$ intersect at a single point exactly when $\|x_1 - x_2\| = \|x_1 - x_3\| = \|x_2 - x_3\| = \sqrt{3} r$ and   Assumption \ref{a1} is not satisfied in this geometric configuration.
For $\|x_1 - x_2\|> r> r_0$ with $r_0 := \|x_1 - x_2\| /\sqrt{3}$, the intersection $B(x_1,r)\cap B(x_2,r)\cap B(x_3,r)$ forms a shape called Reuleaux triangle, whose area is denoted by $R(r)$.
Also, the intersection of two disks of identical radius creates a geometric figure called symmetric lens whose area is denoted by $L(r)$. 
Then it is easy to see that
\begin{align}
\label{reu1} G(\bx,r) &= \vol(A) - \sum_{i=1}^3\vol(B(x_i,r+t\dr)) + 3L(r) 
\quad\text{ if } r< r_0,\\
\label{reu2} G(\bx,r) &= \vol(A) - \sum_{i=1}^3\vol(B(x_i,r+t\dr)) + 3L(r) - 3R(r)
\quad\text{ if } r\geq  r_0.
\end{align}
An explicit calculation using trigonometry yields
$$R(r) = \frac{\pi-\sqrt{3}}{2} s(r)^2 \text{ with } s(r) = \left( r^2 - \frac{3 r_0^2}{4}\right)^{1/2}  -\frac{3 r_0}{2} +r, $$
and
$
R'(r) = (\pi-\sqrt{3}) s(r)s'(r)$, $R''(r) = (\pi-\sqrt{3}) (s'(r)^2 + s(r)s''(r)) 
$
with
\begin{align*}
s'(r) &= r\left( r^2 - \frac{3 r_0^2}{4}\right)^{-1/2}  +1,\quad 
s''(r) = \left( r^2 - \frac{3 r_0^2}{4}\right)^{-1/2}  - r^2\left( r^2 - \frac{3 r_0^2}{4}\right)^{-3/2}. 
\end{align*}
We also compute $s(r_0)=0$, $s'(r_0)=3$ and $s''(r_0)=-6/r_0$.
Thus $R'(r_0)=0$ and
$R''(r_0)=9(\pi-\sqrt{3})\neq 0$. 
Since $R'(r_0)=0$, (\ref{reu1},\ref{reu2}) shows that $G$ has a derivative in direction $(0,0,0,\delta r)$  at $r=r_0$ for any $\dr\in\R$, even though  Assumption~\ref{a1} is not satisfied and we have $\partial_r G(\bx,r_0)=0$.
On the other hand, since $R''(r_0)\neq 0$, $G$ is not twice differentiable at $r=r_0$ in view of (\ref{reu1},\ref{reu2}).
However, (\ref{reu1},\ref{reu2}) shows that $\partial_r G$ is  Gateaux  semidifferentiable  in both directions  $(0,0,0,1)$ and $(0,0,0,-1)$  at $r_0$ with 
\begin{align*}
\lim_{\dr\to 0^-} (\partial_r G(\bx,r_0+\dr) - \partial_r G(\bx,r_0))/\dr &= -6\pi  + 3L''(r_0),\\
\lim_{\dr\to 0^+} (\partial_r G(\bx,r_0+\dr) - \partial_r G(\bx,r_0))/\dr &= -6\pi  + 3L''(r_0) - 3R''(r_0).\\
\end{align*}
\end{example}

\section{Exact calculation of \texorpdfstring{$G$}{G} and its derivatives}\label{sec:calc_algorithms}

In this section, we consider that $A = \cup_{j=1}^{p} A_j$ and $\{A_j\}_{j=1}^{p}$ are non-overlapping convex polygons. (If not available, such decomposition can be computed in $\mathcal{O}(e_A + \bar e_A)$, where $e_A$ is the number of vertices of $A$ and $\bar e_A$ is its number of notches; see, for example, \cite{keil} and the references therein.) The key ingredient for the exact computation of $G$, $\nabla G$, and $\nabla^2 G$ as stated in Section~\ref{sec:shape_opt_covering} is to consider partitions 
\begin{equation} \label{partitiondeverdade}
A_j \cap \Om(\bx,r) = \bigcup_{i \in \mathcal{K}_{A_j}} S_{ij}, \quad j=1,\dots,p,
\end{equation}
where $\mathcal{K}_{A_j} \subseteq \{ i \in \{1,\dots,m\} \;|\; B(x_i,r) \cap A_j \neq \emptyset\}$ 
for $j=1,\dots,p$ and each $S_{ij}$ is such that $\partial S_{ij}$ is a simple and convex curve given by the union of segments and arcs of the circle $\partial B(x_i,r)$. 
It is worth noticing that, since $A_1, A_2,\dots,A_p$ are disjoint, then $\{S_{ij}\}_{(i,j) \in \mathcal{K}}$ with $\mathcal{K} = \{ (i,j) \;|\; j \in \{1,\dots,p\} \mbox{ and } i \in \mathcal{K}_{A_j}\}$ is a partition of $A \cap \Om(\bx,r)$, i.e.,
\begin{equation} \label{xicara}
A \cap \Om(\bx,r) = \bigcup_{(i,j) \in \mathcal{K}} S_{ij},
\end{equation}
see Figure~\ref{fig:Sij}. Note that~8 out of the~10 balls intersect either $A_1$ or $A_2$ in Figure~\ref{fig:Sij}. Let us number the balls intersecting only $A_1$ from~1 to~5 and the balls intersecting only~$A_2$ from 8 to 10. Thus, balls 1 to~5 contribute to~\eqref{partitiondeverdade} with $S_{11}, S_{21}, \dots, S_{51}$, i.e., they contribute to the partition of $A_1$ only; while balls 8, 9, and~10 contribute with $S_{82}$, $S_{92}$, and $S_{10,2}$, i.e., they contribute to the partition of $A_2$ only. Balls 6~and~7 intersect both~$A_1$ and $A_2$ and contribute with $S_{61}$ and $S_{71}$ to the partition of $A_1$ and with $S_{62}$ and $S_{72}$ to the partition of $A_2$. Therefore we have
$\mathcal{K}_{A_1}=\{1,2,3,4,5,6,7\}$, $\mathcal{K}_{A_2}=\{6,7,8,9,10\}$, and $\mathcal{K}=\{ (1,1), (2,1), (3,1), (4,1), (5,1),  (6,1), (6,2), (7,1), (7,2), (8,2), (9,2), (10,2)\}$. In addition, for further use, we define $\mathcal{K}_{B_1}=\mathcal{K}_{B_2}=\mathcal{K}_{B_3}=\mathcal{K}_{B_4}=\mathcal{K}_{B_5}=\{1\}$, $\mathcal{K}_{B_6}=\mathcal{K}_{B_7}=\{1,2\}$, $\mathcal{K}_{B_8}=\mathcal{K}_{B_9}=\mathcal{K}_{B_{10}}=\{2\}$.

\begin{figure}[ht!]
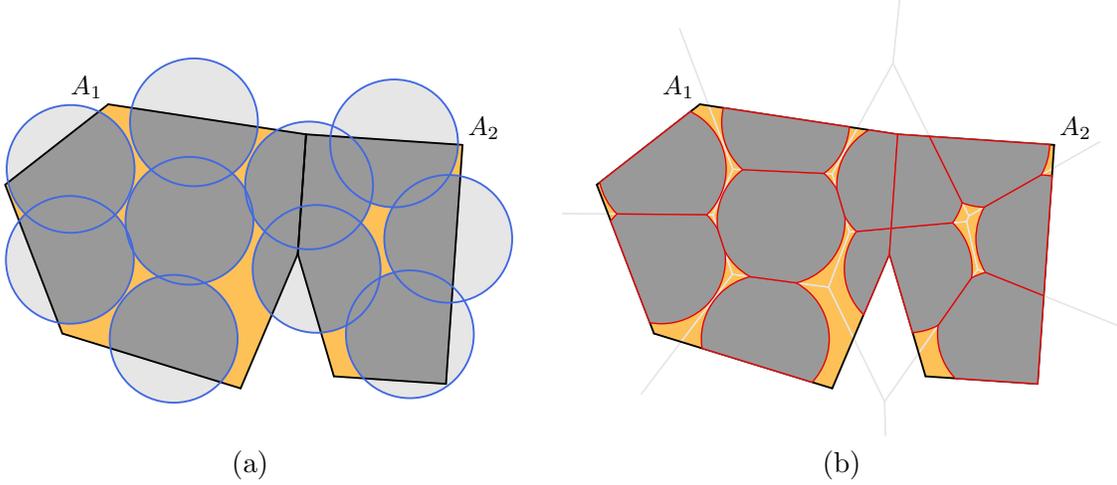

\begin{center}
\begin{tabular}{cc}
\includegraphics{convpols_balls.mps} &
\includegraphics{curv_pols.mps} \\
(a) & (b)
\end{tabular}
\end{center}
\caption{(a) represents a region $A$ to be covered given by $A=\cup_{j=1}^p A_j$ with $p=2$ and an arbitrary configuration of balls $\Om(\bx,r) = \cup_{i=1}^m B(x_i,r)$ with $m=10$. (b) represents the partitions of $A_1 \cap \Om(\bx,r)$ and $A_2 \cap \Om(\bx,r)$ defined in~\eqref{partitiondeverdade} that, together, as expressed in~\eqref{xicara}, represent a partition of $A \cap \Om(\bx,r)$. In (b), the Voronoi diagram that allows the partitions to be computed is depicted.} 
\label{fig:Sij}
\end{figure}

The computation of the partitions in~\eqref{partitiondeverdade} is based on Voronoi diagrams. For a given $(\bx,r)$, we first compute the Voronoi diagram with cells $\{V_i\}_{i=1}^m$ associated with the balls centers $x_1,\dots,x_m$. Each cell $V_i$ is a (bounded or unbounded) polyhedron given by the points $y \in \R^2$ such that $\|y-x_i\| = \min_{\{\ell=1,\dots,m\}} \{ \| y - x_\ell\| \}$. Then, for each $j=1,\dots,p$ and each $i=1,\dots,m$, we compute the convex polygons $W_{ij}=A_j \cap V_i$ and, in the sequence, $S_{ij} = W_{ij} \cap B(x_i,r)$. (Note that, by construction, $W_{ij} \cap B(x_i,r) = W_{ij} \cap \Om(\bx,r)$; and so~\eqref{partitiondeverdade} and, in consequence, \eqref{xicara} hold.) In the construction process, we obtain the sets $\mathcal{K}_{A_j} = \{ i \in \{1,\dots,m\} \;|\; S_{ij} \neq \emptyset \}$, 
$\mathcal{K}_{B_i} = \{ j \in \{1,\dots,p\} \;|\; S_{ij} \neq \emptyset \}$, and $\mathcal{K}$ such that $(i,j) \in \mathcal{K}$ if and only if $S_{ij} \neq \emptyset$. Let $\mathcal{V}(S_{ij})$ be the set of vertices of $S_{ij}$, $\mathcal{A}(S_{ij}) = \partial S_{ij} \cap \partial B(x_i,r)$ the union of the arcs in $\partial S_{ij}$, and $\mathcal{E}(S_{ij}) = \partial S_{ij} \setminus \mathcal{A}(S_{ij})$ the union of the edges in $\partial S_{ij}$. Moreover, we associate with $\mathcal{E}(S_{ij})$ and $\mathcal{A}(S_{ij})$ the corresponding sets of maximal arcs $\mathbb{A}(S_{ij})$ and edges $\mathbb{E}(S_{ij})$. Strictly speaking, these are sets of pairs of points representing arcs and edges, respectively. Each edge is represented by a pair $[v,w]$ of vertices in counter-clockwise order and each arc is represented by a pair $(v,w)$ of vertices, in counter-clockwise order, that unequivocally determines two angles. For each vertex $z \in \mathcal{V}(S_{ij})$, we save whether $z \in \partial A$ or not. If $z \in \partial A$, then we save, whenever it exists, the unitary (Euclidean) norm outward normal vector to $\partial A$ at $z$, named $\nu_A(z)$.  Additionally, for each vertex~$z \in {\cal V}(S_{ij})$, we save the set of indices~$L(z) \subseteq \{1,\dots, m\} \setminus \{i\}$ such that $z \in \partial B(x_\ell,r)$ for all $\ell \in L(z)$. 

Each set $\mathcal{A}_i = \partial B(x_i,r) \cap \partial \Om(\bx,r) \cap A$ for $i=1,\dots,m$, defined in~\eqref{partition}, corresponds to the union of the arcs in $\partial S_{ij}$ for all $j \in \mathcal{K}_{B_i}$, i.e., it holds
\begin{equation} \label{labendita}
\arc_i = \bigcup_{j \in \mathcal{K}_{B_i}} \arc(S_{ij})
\end{equation}
for $i=1,\dots,m$. It is worth noticing~\eqref{labendita} does not mean that every arc in $\pairsvw_i$ belongs to $\pairsvw(S_{ij})$ for some $j \in \mathcal{K}_{B_i}$ nor that $|\pairsvw_i|=\sum_{j \in \mathcal{K}_{B_i}} |\pairsvw(S_{ij})|$. Indeed, if $z$ is an extremity of an arc in $\pairsvw(S_{ij})$ then either $z \in \partial B(x_\ell,r)$ for some $\ell\neq i$ or $z \in \partial A_j$. In the case $z \in \partial A_j$, it may happen that $z \notin \partial A$, and consequently $z$ is not an extremity of an arc in $\pairsvw_i$. To construct $\pairsvw_i$, consecutive arcs (arcs with a extreme in common) in  $\cup_{j \in \mathcal{K}_{B_i}} \pairsvw(S_{ij})$ must be merged into a single arc. So, what holds is that each arc in $\pairsvw_i$ belongs to $\pairsvw(S_{ij})$ for some $j \in \mathcal{K}_{B_i}$ or  is the union of two or more consecutive arcs in $\cup_{j \in \mathcal{K}_{B_i}} \pairsvw(S_{ij})$. Thus $|\pairsvw_i| \leq \sum_{j \in \mathcal{K}_{B_i}} |\pairsvw(S_{ij})|$. The particular case $\arc_i = \partial B(x_i,r)$ is considered separately; in this case, we set $\pairsvw_i=\emptyset$ and $\mathrm{Circle}(\pairsvw_i)$ equal to true.


In a similar way, we also define
\begin{equation} \label{labendita2}
\mathcal{E}_i = \bigcup_{j \in \mathcal{K}_{B_i}} \mathcal{E}(S_{ij}),
\end{equation}
and the associate set $\mathbb{E}_i$ of pairs $[v,w]$ representing edges, for $i=1,\dots,m$. These sets of edges play a role in the computation of $G$ only. Thus, while the same principle of merging consecutive edges could be applied, it has no practical relevance because one way or the other, the same result is obtained.

A second ingredient for the exact computation of $G$ and its derivatives are the parameterizations
\begin{equation} \label{edgespar}
t \mapsto \left( \begin{array}{c} x_{\cal E}(t) \\ y_{\cal E}(t) \end{array} \right) = v + t (w-v), \quad t \in [0,1],
\end{equation}
of each edge represented by $[v,w] \in \cup_{i=1}^m \mathbb{E}_i$; and the parameterizations
\begin{equation} \label{arcspar}
\theta \mapsto \left( \begin{array}{c} x_{\cal A}(\theta) \\ y_{\cal A}(\theta) \end{array} \right) = x_i + r \left( \begin{array}{c} \cos \theta \\ \sin \theta \end{array} \right), \quad \theta \in [\theta_v,\theta_w],
\end{equation}
of each arc represented by $(v,w) \in \pairsvw_i$ for $i=1,\dots,m$, where $\theta_v$ and $\theta_w$ are the angular coordinates of $v-x_i$ and $w-x_i$, respectively.

We are now ready to compute $G$ and its derivatives. By~\eqref{aquiG},
\begin{equation} \label{alg1a}
G(\bx,r) = \vol(A) - \vol(A \cap \Om(\bx,r)) = 
\vol(A) - \sum_{(i,j) \in \mathcal{K}} \vol(S_{ij}).
\end{equation}
By Green's Theorem,
\begin{equation} \label{alg1b}
\vol(S_{ij}) 
= \int_{S_{ij}} dxdy
= \int_{\partial S_{ij}} x \, dy =
\sum_{[v,w] \in \mathbb{E}(S_{ij})} \int_0^1 x_{\mathcal{E}}(t) \, dy_{\mathcal{E}}(t) + 
\sum_{(v,w) \in \pairsvw(S_{ij})} \int_{\theta_v}^{\theta_w} x_{\mathcal{A}}(\theta) \, dy_{\mathcal{A}}(\theta)
\end{equation}
for all $(i,j) \in \mathcal{K}$; while, by~\eqref{edgespar},
\begin{equation} \label{alg1c}
\int_0^1 x_{\mathcal{E}}(t) \, dy_{\mathcal{E}}(t) =
\frac{(v_1 + w_1)(w_2 - v_2)}{2}
\end{equation}
for all $[v,w] \in \mathbb{E}(S_{ij})$ and all $(i,j) \in \mathcal{K}$, and, by~\eqref{arcspar},
\begin{equation} \label{alg1d}
\begin{array}{rcl}
\displaystyle \int_{\theta_v}^{\theta_w} x_{\mathcal{A}}(\theta) \, dy_{\mathcal{A}}(\theta) 
&=& \displaystyle (x_{i})_1 \, r \, (\sin{\theta_w} - \sin{\theta_v})  \\[4mm]
&+& \displaystyle \frac{r^2}{2} \left( \theta_w - \theta_v +\sin{\theta_w} \cos{\theta_w} - \sin{\theta_v} \cos{\theta_v} \right)
\end{array}
\end{equation}
for all $(v,w) \in \pairsvw(S_{ij})$ and all $(i,j) \in \mathcal{K}$. The computation of~$G$ as defined in~\eqref{aquiG} using~(\ref{alg1a}--\ref{alg1d}) is summarized in Algorithm~\ref{G}.

\begin{algorithm}[ht!]
\caption[G]{\textsc{Computes $G(\bx,r)$.}} \label{G}
\KwInput{$\vol(A)$, $(\bx,r)$, and sets $\{\mathbb{E}_i,\pairsvw_i\}_{i=1}^m$.}
\KwOutput{$G(\bx,r)$.}
$\gamma \gets 0$ \\
\For{$i=1,\dots,m$}{
\If{$\mathrm{Circle}(\pairsvw_i)$}{
    $\gamma \gets \gamma + \pi r^2$
    }
\Else{
    \ForEach{$[v,w] \in \mathbb{E}_i$}{
        $\gamma \gets \gamma+\frac{1}{2}(v_1 + w_1)(w_2 - v_2)$
    }
    \ForEach{$(v,w) \in \pairsvw_i$}{
        $\gamma \gets \gamma + 
        {(x_i)}_1 r (\sin \theta_w - \sin \theta_v ) + \frac{r^2}{2} (\theta_w - \theta_v + \sin \theta_w \cos \theta_w - \sin \theta_v \cos \theta_v )$
        }
    }
}
\Return $G = \vol(A) - \gamma$
\end{algorithm}

For computing~$\nabla G = (\partial_{x_1} G(\bx,r),\dots,\partial_{x_m} G(\bx,r),\partial_{r} G(\bx,r))^\top$, by~\eqref{gradG} and \eqref{partition}, we have that
\begin{equation} \label{alg2a}
\begin{array}{rcl}
\partial_{x_i} G(\bx,r) &=& \displaystyle
- \int_{\mathcal{A}_i} \nu_i(z) \, dz 
= \displaystyle - \sum_{(v,w) \in \pairsvw_i} \int_{\theta_v}^{\theta_w} r \, (\cos \theta, \sin \theta)^\top \, d \theta \\[4mm]
&=& \displaystyle \sum_{(v,w) \in \pairsvw_i} r \, (\sin \theta_v - \sin \theta_w, \cos \theta_w - \cos \theta_v)^\top,
\end{array}
\end{equation}
for $i=1,\dots,m$, and, by~\eqref{gradG} and \eqref{unionAi}, we have that
\begin{equation} \label{alg2b}
\partial_{r} G(\bx,r) = 
\displaystyle - \int_{\cup_{i=1}^m \arc_i} dz = 
\displaystyle - \sum_{(v,w) \in \cup_{i=1}^m \pairsvw_i} \int_{\theta_v}^{\theta_w} r \, d \theta = 
\displaystyle - \sum_{(v,w) \in \cup_{i=1}^m \pairsvw_i} r \, (\theta_w - \theta_v).
\end{equation}
The computation of~$\nabla G$ as defined in~\eqref{gradG} using~(\ref{alg2a},\ref{alg2b}) is summarized in Algorithm~\ref{grad_G}.

\begin{algorithm}[ht!]
\caption[$\nabla G$]{\textsc{Computes $\nabla G(\bx,r)$.}} \label{grad_G}
\KwInput{$(\bx,r)$, and sets $\{\pairsvw_i\}_{i=1}^m$.}
\KwOutput{$\nabla G(\bx,r)$.}
$g_r\gets0$ and $g_{x_i}\gets(0, 0)^\top$ for $i=1,\dots,m$.\\
\For{$i=1,\dots,m$}{
\If{$\mathrm{Circle}(\pairsvw_i)$}{
    $g_r \gets g_r - 2 \pi r$
    }
\Else{
    \ForEach{$(v,w) \in \pairsvw_i$}{
        $g_r \gets g_r - r(\theta_w - \theta_v)$ \\
        $g_{x_i} \gets g_{x_i} + r (\sin{\theta_v} - \sin{\theta_w}, \cos{\theta_w} - \cos{\theta_v})^\top$
        }
    }
}
\textbf{return} $\nabla G(\bx,r) = (g_r, g_{x_1}^\top, \dots, g_{x_m}^\top)^\top$
\end{algorithm}

For computing~$\nabla^2 G$, we use that, in~\eqref{d2Grr},
\begin{equation} \label{alg3a}
- \frac{\mathrm{Per}(\partial\Om(\bx,r)\cap A)}{r} = 
\sum_{(v,w) \in \cup_{i=1}^m \pairsvw_i} (\theta_v - \theta_w),
\end{equation}
in~\eqref{d2xixiG}, 
\begin{align} \label{alg3b}
& \frac{1}{r} \displaystyle \int_{\arc_i} - \nu_i(z) \otimes \nu_i(z) + \tau_i(z) \otimes \tau_i(z) \, dz \nonumber \\ 
&= \frac{1}{r} \sum_{(v,w) \in \pairsvw_i} r \int_{\theta_v}^{\theta_w} - \left(
\begin{array}{cc}
(\cos \theta)^2 & \sin \theta \cos \theta \\ 
\sin \theta \cos \theta & 
(\sin \theta)^2
\end{array}
\right) + \left(
\begin{array}{cc}
(\sin \theta)^2 & -\sin \theta \cos \theta \\ 
-\sin \theta \cos \theta & 
(\cos \theta)^2
\end{array}
\right) \, d \theta \\
&= \sum_{(v,w) \in \pairsvw_i} 
\left(
\begin{array}{cc}
\sin(\theta_v - \theta_w)\cos(\theta_v + \theta_w) & (\cos\theta_w)^2 - (\cos\theta_v)^2 \\ 
(\cos\theta_w)^2 - (\cos\theta_v)^2 & 
\sin(\theta_w - \theta_v)\cos(\theta_v + \theta_w)
\end{array}
\right),  \nonumber 
\end{align}
and, in~\eqref{eq:781},
\begin{equation} \label{alg3c}
- \frac{1}{r} \int_{\arc_i} \nu_i(z) \, dz\\[4mm]
= \sum_{(v,w) \in \pairsvw_i} (\sin \theta_v - \sin \theta_w, \cos \theta_w - \cos \theta_v)^\top.
\end{equation}
Recall that in~(\ref{d2Grr},\ref{d2xixiG},\ref{eq:781}), for $z \in \partial B(x_i,r)$, $\nu_{-i}(z)$ represents the unitary-norm outwards normal vector to the set intersecting $\partial B(x_i,r)$ at $z$. If this set is $\partial A$, then $\nu_{-i}(z)=\nu_A(z)$. If this set is $\partial B(x_\ell,r)$ for some $\ell \in L(z)$, then $\nu_{-i}(z)=\nu_\ell(z)=(\cos \vartheta_z, \sin \vartheta_z)^\top$, where $\vartheta_z$ is the angular coordinate of $z-x_\ell$.
With these definitions and substituting~(\ref{alg3a},\ref{alg3b},\ref{alg3c}) in~(\ref{d2Grr},\ref{d2xixiG},\ref{d2Gxixj},\ref{eq:781}), we arrive at Algorithm~\ref{hessian_G}.  

\begin{algorithm}[ht!]
\caption[hessian]{\textsc{Computes $\nabla^2 G(\bx,r)$.}} \label{hessian_G}
\KwInput{$(\bx,r)$ and sets $\{\pairsvw_i\}_{i=1}^m$.}
\KwOutput{The lower triangle of $H = \nabla^2 G(\bx,r) \in \mathbb{R}^{2m+1,2m+1}$.}
$H \gets 0$.\\
\For{$i=1,\dots,m$}{
\If{$\mathrm{Circle}(\pairsvw_i)$}{
    $h_{2m+1, 2m+1} \gets h_{2m+1, 2m+1} - 2\pi$
    }
\Else{
\ForEach{$(v,w) \in \pairsvw_i$}{
\vspace{1mm}
let $a \odot b$ mean $a \gets a + b$\\
$\scalebox{0.90}{$
\begin{array}{rcl}
\left(
\begin{array}{cc}
h_{2i-1, 2i-1} & \\ 
h_{2i,   2i-1} & h_{2i,   2i}
\end{array}
\right)
&\odot&
\left(
\begin{array}{cc}
\sin(\theta_v - \theta_w)\cos(\theta_v + \theta_w) & \\ 
(\cos\theta_w)^2 - (\cos\theta_v)^2 & 
\sin(\theta_w - \theta_v)\cos(\theta_v + \theta_w)
\end{array}
\right)
\\[4mm]
(h_{2m+1, 2i-1}, h_{2m+1, 2i})
&\odot&
(\sin \theta_v - \sin \theta_w, \cos \theta_w - \cos \theta_v)
\\[2mm]
h_{2m+1, 2m+1}
&\odot&
\theta_v - \theta_w
\end{array}
$}$\\
\For{$z \in \{v,w\}$}{
\vspace{1mm}
\textbf{if} $z=v$ \textbf{then} let $a \odot b$ mean $a \gets a - b$ \textbf{else} let $a \odot b$ mean $a \gets a + b$\\[1mm]
\If{$z \in \partial A$}{
\vspace{1mm}
$\scalebox{0.90}{$
\begin{array}{l}
\alpha \gets \left(-\nu_A(z) \cdot (\cos \theta_z, \sin \theta_z)^\top \right) / \left( \nu_A(z) \cdot (-\sin \theta_z, \cos \theta_z)^\top \right) \\[2mm]
\begin{array}{rcl}
\left(
\begin{array}{cc}
h_{2i-1, 2i-1} & \\
h_{2i, 2i-1} & h_{2i, 2i}
\end{array}
\right)
&\odot&
\alpha
\left(
\begin{array}{cc}
(\cos\theta_z)^2 & \\
\sin \theta_z \cos \theta_z & (\sin\theta_z)^2
\end{array}
\right)
\\[4mm]
(h_{2m+1, 2i-1}, h_{2m+1, 2i}) &\odot& \alpha(\cos \theta_z, \sin \theta_z)
\\[2mm]
h_{2m+1, 2m+1} &\odot& \alpha 
\end{array}
\end{array}
$}$
}
\ForEach{$\ell(z) \in L(z)$}{
\vspace{1mm}
$\scalebox{0.90}{$
\begin{array}{rcl}
\left(
\begin{array}{cc}
h_{2i-1, 2i-1} & \\ 
h_{2i,   2i-1} & h_{2i,   2i}
\end{array}
\right)
&\odot&
\left(
\begin{array}{cc}
\cotan(\vartheta_z - \theta_z) (\cos\theta_z)^2 & \\ 
\cotan(\vartheta_z - \theta_z) \sin \theta_z \cos \theta_z & 
\cotan(\vartheta_z - \theta_z) (\sin\theta_z)^2
\end{array}
\right)
\\[4mm]
h_{2m+1, 2i-1}
&\odot&
\cotan(\vartheta_z - \theta_z) \cos \theta_z 
- \cos \theta_z / \sin(\vartheta_z - \theta_z)
\\[2mm]
h_{2m+1, 2i}
&\odot&
\cotan(\vartheta_z - \theta_z) \sin \theta_z 
- \sin \theta_z / \sin(\vartheta_z - \theta_z)
\\[2mm]
h_{2m+1, 2m+1}
&\odot&
(\cos(\vartheta_z - \theta_z)-1) / \sin(\vartheta_z - \theta_z) 
\end{array}
$}$\\
\If{${\ell(z)} > i$}{
$\scalebox{0.90}{$
\begin{array}{rcl}
\left(
\begin{array}{cc}
h_{2\ell(z)-1, 2i-1} & h_{2\ell(z)-1, 2i}\\ 
h_{2\ell(z),   2i-1} & h_{2\ell(z),   2i}
\end{array}
\right)
&\odot&
-(\sin(\vartheta_z - \theta_z))^{-1}
\left(
\begin{array}{cc}
\cos \theta_z \cos \vartheta_z & 
\sin \theta_z \cos \vartheta_z \\ 
\cos \theta_z \sin \vartheta_z & 
\sin \theta_z \sin \vartheta_z
\end{array}
\right)
\end{array}
$}$
}
}
}
}
}
}
\Return $H$
\end{algorithm}

Algorithms~\ref{G}, \ref{grad_G}, and \ref{hessian_G} depend on the computation of sets $\mathbb{E}_i$ and $\mathbb{A}_i$ for $i=1,\dots,m$. Computing these sets requires (a) to compute the Voronoi diagram with cells $\{V_i\}_{i=1}^m$ associated with the balls' centers $x_1,\dots,x_m$ and (b) for each $i \in \{1,\dots,m\}$ and $j \in \{1,\dots,p\}$, to compute $W_{ij} = V_i \cap A_j$ and $S_{ij} = W_{ij} \cap B(x_i,r)$. Computing the Voronoi diagram, using for example Fortune's algorithm~\cite{fortune}, has known time complexity $\mathcal{O}(m\log{m})$~\cite[Lem.~7.9, p.158]{deberg}. Since the intersection between a two-dimensional polyhedron defined by~$a$ half-planes and a convex polygon with~$b$ sides can be computed in $\mathcal{O}(ab)$~\cite{horowitz1992}, all $W_{ij}$ can be computed in
\begin{equation} \label{primc}
\mathcal{O}(\sum_{i=1}^m{\sum_{j=1}^p}{e_{V_i} e_{A_j}}),
\end{equation}
where $e_{V_i}$ is the number of half-planes that define $V_i$, for $i=1,\dots,m$, and $e_{A_j}$ is the number of sides of each $A_j$, for $j=1,\dots,p$. However, it is also known~\cite[Thm.7.3, p.150]{deberg} that a Voronoi diagram generated by $m\geq3$ points has at most $3m-6$ edges; and since each edge is part of exactly two cells, we have that $\sum_{i=1}^me_{V_i}=\mathcal{O}(m)$. Thus, \eqref{primc} reduces to $\mathcal{O}(m \sum_{j=1}^p e_{A_j})$. By construction, it also holds that $\sum_{i=1}^m{|\mathbb{E}_i|}$ is $\mathcal{O}(m \sum_{j=1}^p e_{A_j})$. Finally, a simple inspection of Algorithm~\ref{inter_pol_circle}, used to compute $S_{ij} = W_{ij} \cap B(x_i,r)$, shows that the computational effort required to compute all $S_{ij}$, as well as $\sum_{i=1}^m{|\mathbb{A}_i|}$, are both $\mathcal{O}(m \sum_{j=1}^p e_{A_j})$. This implies that the worst-case time complexity of Algorithms~\ref{G}, \ref{grad_G}, and \ref{hessian_G} is $\mathcal{O}(m\log{m} + m \sum_{j=1}^p e_{A_j})$.

\section{Numerical experiments} \label{sec:num_exp}

In this section, we aim to illustrate the capabilities and limitations of the proposed approach. We implemented Algorithms~\ref{G}, \ref{grad_G}, and \ref{hessian_G} in Fortran~90. Given the balls' centers $\{x_i\}_{i=1}^m$, the Voronoi diagram is computed with subroutine \textsc{Dtris2} from Geompack~\cite{geompack} (available at \url{https://people.math.sc.edu/Burkardt/f_src/geompack2/geompack2.html}). In fact, \textsc{Dtris2} provides a Delaunay triangulation from which the Voronoi diagram is extracted. The intersection $W_{ij}$ of each Voronoi cell $V_i$ (that is a bounded or unbounded polyhedron) and each convex polygon~$A_j$ is computed with the Sutherland-Hodgman algorithm~\cite{sutherland1974}. For each convex polygon~$W_{ij}$, the intersection $S_{ij}$ with the ball $B(x_i,r)$ is computed with an adaptation of a single iteration of the Sutherland-Hodgman algorithm, detailed as Algorithm~\ref{inter_pol_circle} in Appendix~\ref{appen:alg4}.

Problem~\eqref{prob1} is a nonlinear programming problem of the form
\begin{equation} \label{prob2}
\mbox{Minimize } f(\bx,r) := r \mbox{ subject to } G(\bx,r)=0 \mbox{ and } r \geq 0
\end{equation} 
that can be tackled with an Augmented Lagrangian (AL) approach~\cite{bmbook}. In the numerical experiments, we considered the AL method Algencan~\cite{abmstango,bmbook,bmcomper}. Algencan~4.0, implemented in Fortran~90 and available at~\url{http://www.ime.usp.br/~egbirgin/tango/}, was considered. Algencan is an AL method with safeguards that, at each iteration, solves a bound-constrained subproblem. Since, in the present work, second-order derivatives are available, subproblems are solved with an active-set Newton's method; see~\cite{bmgencan} and~\cite[Ch.9]{bmbook} for details. When Algencan is applied to problem~(\ref{prob2}), on success, it finds $(\bx^\star,r^\star,\lambda^\star)$ with $r^\star > 0$ satisfying
\begin{equation} \label{kkt}
\| \nabla f(\bx^\star,r^\star) + \lambda^\star \nabla G(\bx^\star,r^\star) \|_{\infty} 
\leq \varepsilon_{\mathrm{opt}} \mbox{ and }
\| G(\bx^\star,r^\star) \|_{\infty} \leq \varepsilon_{\mathrm{feas}},
\end{equation}
where $\varepsilon_{\mathrm{feas}} > 0$ and $\varepsilon_{\mathrm{opt}} > 0$ are given feasibility and optimality tolerances, respectively; i.e., it finds a point that approximately satisfies KKT conditions for problem~(\ref{prob2}). Following~\cite{coveringfirst}, in order to enhance the probability of finding an approximation to a global minimizer, a simple multistart strategy with random initial guesses is employed; see~\cite[\S5]{coveringfirst} for details. In the numerical experiments, we considered $\varepsilon_{\mathrm{feas}} = \varepsilon_{\mathrm{opt}} = 10^{-8}$. 

In the numerical experiments, we considered (i) a non-convex polygon with holes already considered in~\cite{stoyan}, (ii) a sketch of a map of America available from~\cite[\S13.2]{bmbook} and already considered in~\cite{coveringfirst}, (iii) an eight-pointed star, (iv) iteration two of the Minkowski island fractal, and (v) iteration three of the Ces\`aro fractal; see Figures~\ref{fig:stoyan}a--\ref{fig:cesaro}a. In Figures~\ref{fig:stoyan}b--\ref{fig:cesaro}b, the way in which the problems were partitioned into convex polygons is made explicit. Appendix~\ref{appen:probs} presents an explicit description of each problem by exhibiting the vertices of each convex polygon that compose the problem.

Fortran source code of Algorithms~\ref{G}, \ref{grad_G}, \ref{hessian_G}, and~\ref{inter_pol_circle}, the source code of the considered problems, as well as the source code necessary to reproduce all numerical experiments, is available at \url{http://www.ime.usp.br/~egbirgin/}. All tests were conducted on a computer with an AMD Opteron 6376 processor and 256GB 1866 MHz DDR3 of RAM memory, running Debian GNU/Linux (version 9.13--stretch). Code was compiled by the GFortran compiler of GCC (version 6.3.0) with the -O3 optimization directive enabled.

In the experiments, we covered the five considered regions with $m \in\{10, 20, \dots, 100 \}$ balls. For each problem and each considered value of $m$, the multistart strategy makes $10{,}000$ attempts, i.e.\ $10{,}000$ different random initial guesses are considered. Table~\ref{tab:results} and Figures~\ref{fig:stoyan}--\ref{fig:cesaro} show the results. In Table~\ref{tab:results}, $r^*$ represents the smallest obtained radius, $G(\bx^*,r^*)$ corresponds to the value of $G$ at the obtained solution, and ``trial'' is the ordinal of the initial guess that yields the smallest radius. In addition, some performance metrics are also displayed in the remaining columns of the table. ``outit'' and ``innit'' correspond to the so called outer and inner iterations of the AL method, respectively, ``Alg.1'', ``Alg.2'', and ``Alg.3'' correspond to the number of calls to Algorithms~\ref{G}, \ref{grad_G}, and~\ref{hessian_G}, respectively, i.e.\ to the number of evaluations of~$G$, $\nabla G$, and $\nabla^2 G$ that were required in the optimization process, and ``CPU Time'' corresponds to the elapsed CPU time in seconds. All these performance metrics correspond to the trial that leads to the smallest radius for a given problem and a given number of balls~$m$. Thus, the whole process took approximately $10{,}000$ times this effort. Clearly, the overall cost can be reduced by reducing the number of trials. Figure~\ref{bestr} illustrates, for the ``non-convex with holes problem'' with $m \in \{10, 20, \dots, 100\}$, the best obtained radius as a function of the number of trials. The picture shows that, for all values of $m$, good quality local minimizers are found with less than 100 trials and that in the remaining 99\% additional trials only marginal improvements are obtained.

\begin{figure}[ht!]
\begin{center}
\includegraphics{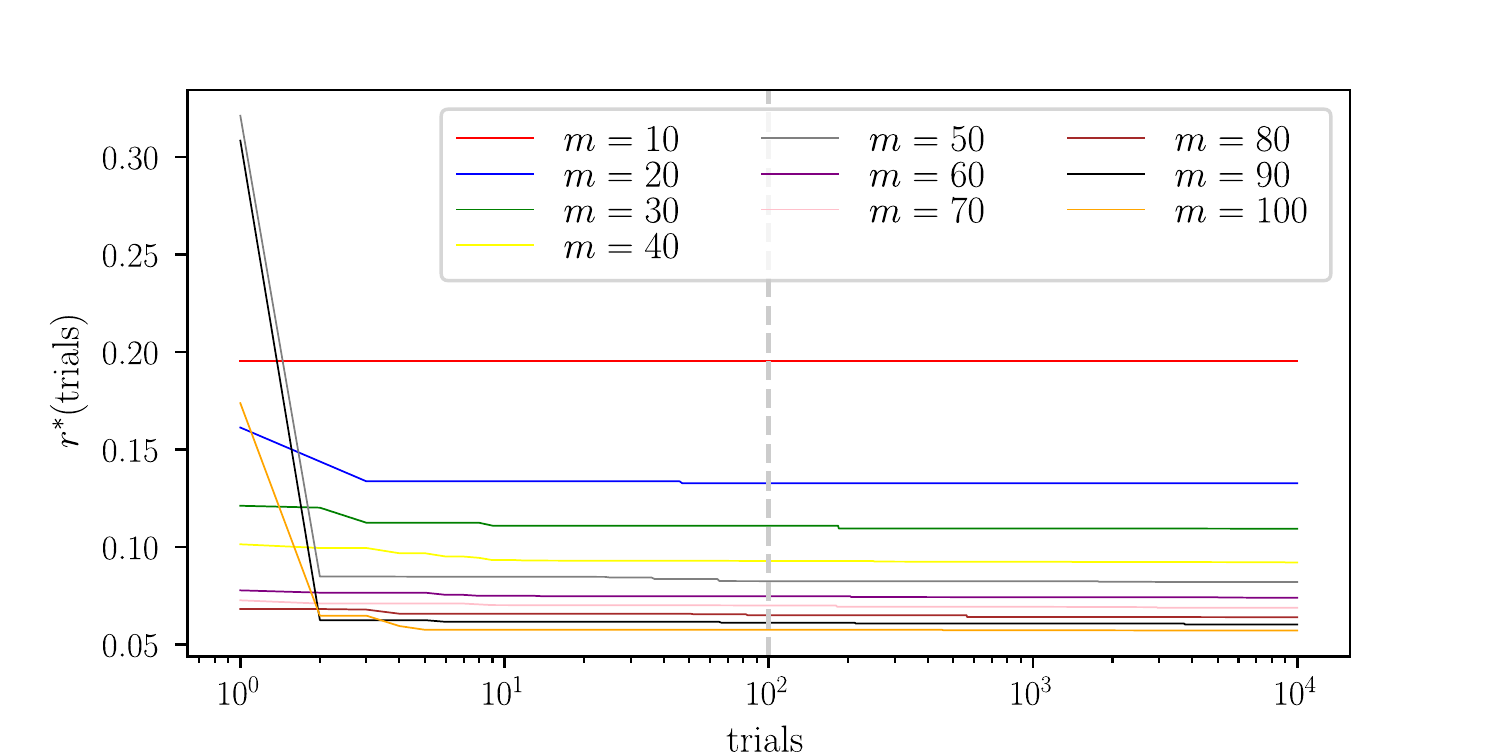}
\end{center}
\caption{Best radius $r^*$ for $m \in \{10, 20, \dots, 100\}$ as a function of the number of trials in the multistart globalization strategy.}
\label{bestr}
\end{figure}

As a whole, numerical experiments show that, by using second-order information, the AL method is able to find high-precision local solutions efficiently. It is worth noticing that, as shown in column $G(\bx^*,r^*)$ of Table~\ref{tab:results}, using $\varepsilon_{\mathrm{feas}}=10^{-8}$ means that the area of the region~$A$ to be covered and the covered region $A \cap \Om(\bx^*,r^*)$ coincide in eight significant digits. Since, in the considered problems, the region with largest area has area equal to~16 (see the description of the problem in Appendix~\ref{appen:probs}), this means that reported solutions cover more than 99.999999\% of the region. This precision is in contrast with the relatively low-quality solutions obtained with the approximate procedure considered in~\cite{coveringfirst}. A scaled versions of the non-convex region with holes considered in the present work was also considered in~\cite{stoyan}, where radius $r^* = 16.6176655 / 150 \approx 0.110784446$ and $r^*= 14.07100757 / 150 \approx 0.09380671713$ for the cases with $m=30$ and $m=40$ were reported. A direct comparison is not possible, because the balls' centers and the covering's precision of these solutions was not reported in~\cite{stoyan}. Anyway, smaller radii were found for these two cases in the present work, namely, $r^* = 0.10944963099046681$ and $r^*=0.092110416532448419$, respectively. The region that represents a sketch of the map of America was also considered in~\cite{coveringfirst}. Solutions presented in~\cite{coveringfirst} are not comparable to the ones presented here. The latter are much more precise and can be found with much less effort. 

To put the practical performance of the current approach in perspective in relation to the practical performance of the method implemented in~\cite{coveringfirst}, consider the trivial configuration depicted in Figure~\ref{fig:trivial}. The configuration shows a square of side three with the bottom-left corner at the origin and two unitary-radius balls with centers $x_1=(0,3)^\top$ and $x_2=(1.2, 1.7)^\top$. The covered area can be computed analytically and it is given by $\mathrm{Vol}(A \cap \Omega(\bx,r)) = 5\pi/4 - 2\arccos(d/2) + d \sqrt{1 - (d/2)^2} \approx 3.781718647855564$, where $d=\|x_1-x_2\|$. Algorithm~\ref{G} computes this quantity up to the machine precision in $10^{-6}$ seconds of CPU time. Algorithm~1 from~\cite{coveringfirst}, devised to cover more general non-polygonal regions, approximates a covered area with precision $O(h)$ at cost $O(h^2)$ by partitioning a region $D$ that contains $A$ in small squares of side~$h$, where $h>0$ is a given parameter. In this specific trivial example, it takes 271.92 seconds of CPU time to compute the covered area with half of the machine precision using $h=10^{-5}$. (With $h=10^{-3}$ and $h=10^{-4}$, four and six correct decimal digits are obtained, by consuming 0.024 and 2.4 seconds of CPU time, respectively. Moreover, the cost is proportional to the area of $D$, which is as small as possible since we considered $D=A$.) So, in this trivial example, we showed that the approach proposed in the present work computes the covered area with twice the number of correct digits with a computational cost that is eight orders of magnitude smaller (i.e.,\ a hundred million times faster) than the cost of the approach proposed in~\cite{coveringfirst}, thus dramatically improving the computational efficiency.
This, together with a similar state of things with respect to the computation of~$\nabla G$, plus the computation of $\nabla^2 G$ that is absent in~\cite{coveringfirst}, justify the much higher quality of the obtained results.

\begin{figure}[ht!]
\centering
\includegraphics{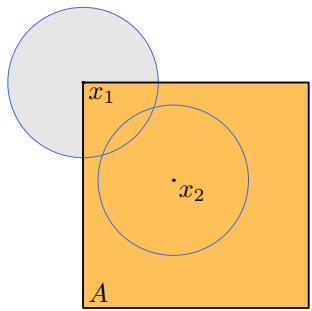}
\caption{A trivial example that illustrates the comparison between the exact computation of~$G$ introduced in the current work and the approximate scheme considered in~\cite{coveringfirst}.}
\label{fig:trivial}
\end{figure}



\begin{table}[ht!]
\begin{center}
\resizebox{0.9\textwidth}{!}{
\begin{tabular}{|c|c|c|c|r|r|r|r|r|r|c|}
\hline
& $m$ & $r^*$ & $G(\bx^*,r^*)$ &
\multicolumn{1}{c|}{trial} & \multicolumn{1}{c|}{outit} & \multicolumn{1}{c|}{innit} & 
\multicolumn{1}{c|}{Alg.1} & \multicolumn{1}{c|}{Alg.2} & \multicolumn{1}{c|}{Alg.3} &
\multicolumn{1}{c|}{CPU Time} \\
\hline
\hline
\multirow{10}{*}{\rotatebox{90}{Non-convex with holes}}
&10	 & 	1.9546630973359513e$-$01 &  5.2e$-$09	 & 	7078	 & 	23	 & 	154	 & 	538	 & 	388	 & 	384	 & 	0.33 \\
&20	 & 	1.3277721146997093e$-$01	 & 	4.2e$-$09	 & 	4580	 & 	21	 & 	123	 & 	426	 & 	345	 & 	333	 & 	0.56\\
&30	 & 	1.0944963099046681e$-$01	 & 	9.9e$-$09	 & 	7155	 & 	22	 & 	187	 & 	1154	 & 	413	 & 	407	 & 	1.48\\
&40	 & 	9.2110416532448419e$-$02	 & 	9.3e$-$09	 & 	8981	 & 	21	 & 	209	 & 	847	 & 	432	 & 	419	 & 	1.85\\
&50	 & 	8.2059696677895658e$-$02	 & 	9.0e$-$09	 & 	3176	 & 	21	 & 	218	 & 	937	 & 	450	 & 	428	 & 	2.57\\
&60	 & 	7.3972529936974535e$-$02	 & 	8.4e$-$09	 & 	7718	 & 	22	 & 	245	 & 	1750	 & 	484	 & 	465	 & 	4.54\\
&70	 & 	6.8954683287629770e$-$02	 & 	9.0e$-$09	 & 	2942	 & 	20	 & 	209	 & 	1228	 & 	421	 & 	409	 & 	4.35\\
&80	 & 	6.4065368587975027e$-$02	 & 	7.5e$-$09	 & 	8908	 & 	21	 & 	209	 & 	1366	 & 	419	 & 	419	 & 	5.69\\
&90	 & 	6.0345840506149377e$-$02	 & 	7.7e$-$09	 & 	3741	 & 	23	 & 	263	 & 	2595	 & 	500	 & 	493	 & 	9.71\\
&100	 & 	5.7226511303503126e$-$02	 & 	6.9e$-$09	 & 	2619	 & 	20	 & 	225	 & 	1390	 & 	448	 & 	425	 & 	5.45\\
\hline
\multirow{10}{*}{\rotatebox{90}{Sketch of America map}}
&10	 & 	1.1022680217297048e$-$01	 & 	6.2e$-$09	 & 	7191	 & 	22	 & 	226	 & 	1198	 & 	434	 & 	446	 & 	0.91\\
&20	 & 	7.0566193751253600e$-$02	 & 	4.2e$-$09	 & 	558	 & 	21	 & 	256	 & 	1541	 & 	455	 & 	466	 & 	2.14\\
&30	 & 	5.6728945376199408e$-$02	 & 	3.7e$-$09	 & 	3341	 & 	20	 & 	241	 & 	1451	 & 	428	 & 	441	 & 	3.36\\
&40	 & 	4.8479681841390981e$-$02	 & 	5.3e$-$09	 & 	7518	 & 	21	 & 	274	 & 	1227	 & 	506	 & 	484	 & 	4.18\\
&50	 & 	4.3079623896669902e$-$02	 & 	4.6e$-$09	 & 	9471	 & 	22	 & 	190	 & 	915	 & 	405	 & 	410	 & 	3.82\\
&60	 & 	3.8669223381267957e$-$02	 & 	9.0e$-$09	 & 	6539	 & 	22	 & 	328	 & 	2124	 & 	544	 & 	548	 & 	9.17\\
&70	 & 	3.5479536239229441e$-$02	 & 	9.3e$-$09	 & 	2774	 & 	20	 & 	290	 & 	1864	 & 	508	 & 	490	 & 	10.81\\
&80	 & 	3.3035213466515133e$-$02	 & 	3.7e$-$09	 & 	9176	 & 	23	 & 	281	 & 	1098	 & 	529	 & 	511	 & 	8.94\\
&90	 & 	3.1081859427563651e$-$02	 & 	9.4e$-$09	 & 	1815	 & 	20	 & 	296	 & 	967	 & 	528	 & 	496	 & 	11.20\\
&100	 & 	2.9185582405640495e$-$02	 & 	7.3e$-$09	 & 	2427	 & 	21	 & 	302	 & 	1271	 & 	525	 & 	512	 & 	10.55\\
\hline
\multirow{10}{*}{\rotatebox{90}{Eight-pointed star}}
&10	 & 	1.3040713549156926e$+$00	 & 	7.4e$-$09	 & 	2129	 & 	28	 & 	212	 & 	1405	 & 	471	 & 	492	 & 	0.40\\
&20	 & 	7.2447962534018184e$-$01	 & 	6.6e$-$09	 & 	1569	 & 	28	 & 	383	 & 	3437	 & 	682	 & 	663	 & 	1.71\\
&30	 & 	5.5386599521018731e$-$01	 & 	4.4e$-$09	 & 	9204	 & 	28	 & 	241	 & 	971	 & 	539	 & 	521	 & 	1.21\\
&40	 & 	4.6618323934219452e$-$01	 & 	4.1e$-$09	 & 	9298	 & 	28	 & 	312	 & 	1999	 & 	614	 & 	592	 & 	2.68\\
&50	 & 	4.1522639848076626e$-$01	 & 	3.7e$-$09	 & 	759	 & 	27	 & 	276	 & 	1974	 & 	572	 & 	546	 & 	3.51\\
&60	 & 	3.7211553871395336e$-$01	 & 	1.0e$-$08	 & 	8549	 & 	27	 & 	278	 & 	2235	 & 	568	 & 	541	 & 	4.74\\
&70	 & 	3.3883252892004639e$-$01	 & 	9.5e$-$09	 & 	3297	 & 	26	 & 	247	 & 	818	 & 	545	 & 	507	 & 	3.53\\
&80	 & 	3.1591211839929362e$-$01	 & 	8.9e$-$09	 & 	3712	 & 	26	 & 	266	 & 	1160	 & 	559	 & 	526	 & 	4.46\\
&90	 & 	2.9594385965306919e$-$01	 & 	8.8e$-$09	 & 	257	 & 	27	 & 	309	 & 	3063	 & 	613	 & 	579	 & 	10.15\\
&100	 & 	2.7907469799758938e$-$01	 & 	8.4e$-$09	 & 	8809	 & 	26	 & 	274	 & 	2596	 & 	540	 & 	533	 & 	7.12\\
\hline
\multirow{10}{*}{\rotatebox{90}{Minkowski island fractal}}
&10	 & 	9.9730787966959566e$-$01	 & 	5.8e$-$09	 & 	85	 & 	28	 & 	269	 & 	2287	 & 	490	 & 	549	 & 	0.79\\
&20	 & 	6.4157361024666815e$-$01	 & 	5.4e$-$09	 & 	114	 & 	29	 & 	233	 & 	943	 & 	533	 & 	523	 & 	1.03\\
&30	 & 	5.3259264476359935e$-$01	 & 	4.9e$-$09	 & 	9963	 & 	26	 & 	303	 & 	1059	 & 	587	 & 	563	 & 	1.74\\
&40	 & 	4.4275330752709730e$-$01	 & 	4.0e$-$09	 & 	9678	 & 	27	 & 	323	 & 	3034	 & 	587	 & 	593	 & 	4.22\\
&50	 & 	3.9534726462521569e$-$01	 & 	9.7e$-$09	 & 	3428	 & 	24	 & 	230	 & 	1155	 & 	479	 & 	470	 & 	2.82\\
&60	 & 	3.4918562471568843e$-$01	 & 	8.8e$-$09	 & 	8144	 & 	26	 & 	248	 & 	775	 & 	522	 & 	508	 & 	3.15\\
&70	 & 	3.2807457983514665e$-$01	 & 	9.3e$-$09	 & 	4385	 & 	27	 & 	278	 & 	2414	 & 	549	 & 	548	 & 	7.06\\
&80	 & 	3.1016946802464157e$-$01	 & 	9.6e$-$09	 & 	7306	 & 	27	 & 	298	 & 	3068	 & 	570	 & 	568	 & 	9.58\\
&90	 & 	2.9050989196451837e$-$01	 & 	8.6e$-$09	 & 	9902	 & 	26	 & 	346	 & 	3204	 & 	591	 & 	606	 & 	12.18\\
&100	 & 	2.7525512468934971e$-$01	 & 	9.0e$-$09	 & 	7719	 & 	26	 & 	377	 & 	2929	 & 	672	 & 	637	 & 	10.51\\
\hline
\multirow{10}{*}{\rotatebox{90}{Cesàro fractal}}
&10	 & 	2.1276864595120507e$-$01	 & 	5.5e$-$09	 & 	7054	 & 	22	 & 	180	 & 	1348	 & 	377	 & 	400	 & 	0.63\\
&20	 & 	1.3326878209070328e$-$01	 & 	3.8e$-$09	 & 	4870	 & 	23	 & 	278	 & 	1152	 & 	421	 & 	508	 & 	1.22\\
&30	 & 	1.0522163653090458e$-$01	 & 	4.1e$-$09	 & 	7850	 & 	23	 & 	245	 & 	1219	 & 	486	 & 	475	 & 	2.18\\
&40	 & 	9.3428035096055656e$-$02	 & 	9.4e$-$09	 & 	2646	 & 	21	 & 	193	 & 	871	 & 	424	 & 	403	 & 	2.43\\
&50	 & 	8.3314180748730718e$-$02	 & 	9.4e$-$09	 & 	1317	 & 	23	 & 	197	 & 	1322	 & 	441	 & 	427	 & 	3.68\\
&60	 & 	7.8415153849036370e$-$02	 & 	8.8e$-$12	 & 	9229	 & 	31	 & 	433	 & 	3704	 & 	674	 & 	743	 & 	10.83\\
&70	 & 	7.0460470988540802e$-$02	 & 	8.4e$-$09	 & 	5859	 & 	22	 & 	289	 & 	1698	 & 	544	 & 	509	 & 	6.96\\
&80	 & 	6.6110791995596219e$-$02	 & 	8.6e$-$09	 & 	7697	 & 	21	 & 	329	 & 	1694	 & 	581	 & 	539	 & 	9.38\\
&90	 & 	6.1956278506660224e$-$02	 & 	7.8e$-$09	 & 	5722	 & 	23	 & 	318	 & 	3185	 & 	568	 & 	548	 & 	15.12\\
&100	 & 	5.8465961897078852e$-$02	 & 	8.6e$-$09	 & 	3205	 & 	21	 & 	296	 & 	1729	 & 	548	 & 	506	 & 	8.32\\
\hline
\end{tabular}}
\end{center}
\caption{Details of the obtained solutions and performance metrics of the application of Algencan to the five considered covering problems.}
\label{tab:results}
\end{table}

\begin{figure}[ht!]
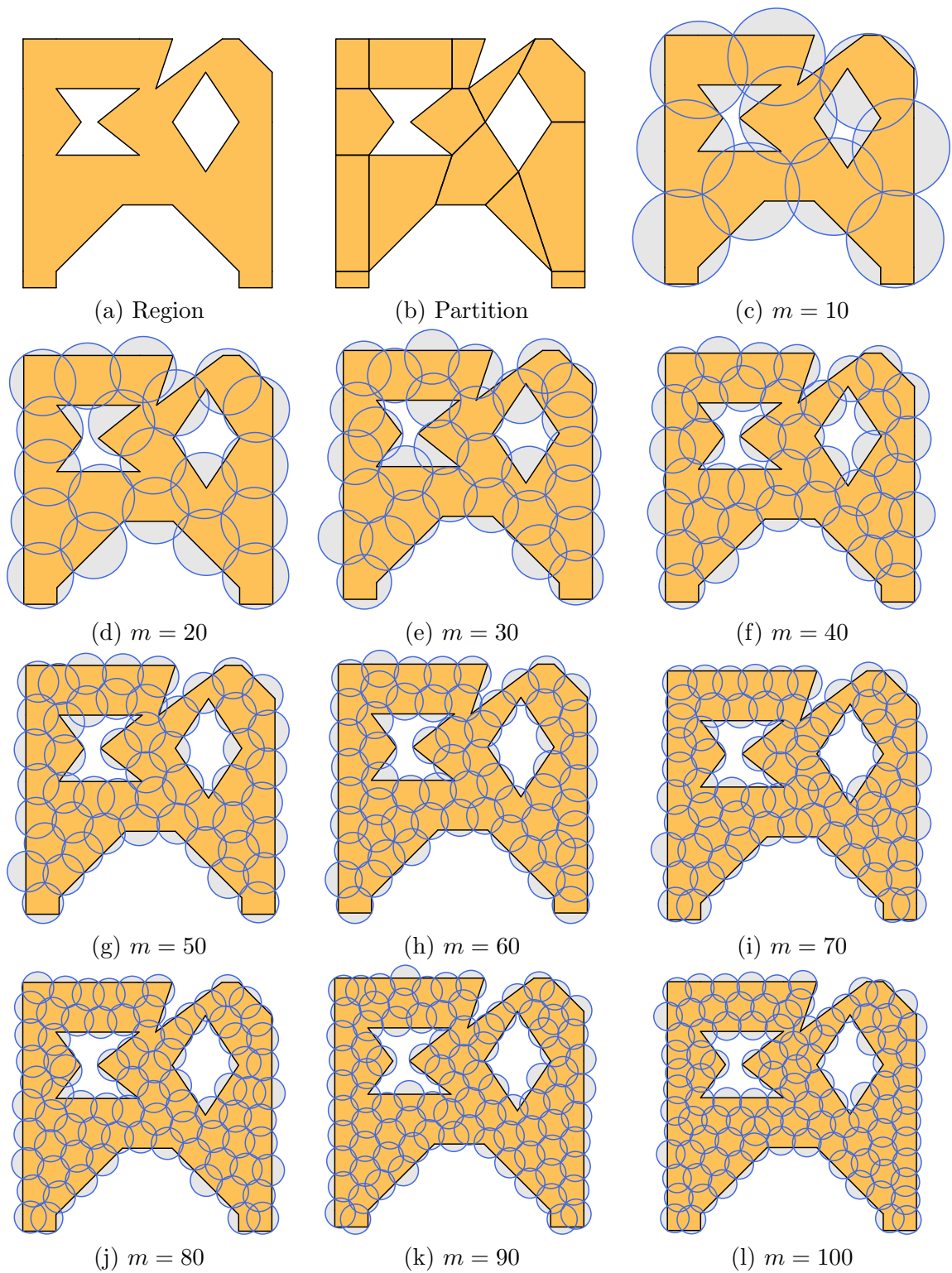

\begin{center}
\begin{tabular}{ccc}
\includegraphics[scale=0.8]{region-stoyan.mps} &
\includegraphics[scale=0.8]{stoyanpartition.mps} &
\includegraphics[scale=0.8]{covering-stoyan10-10.mps} \\
(a) Region & (b) Partition & (c) $m=10$ \\
\includegraphics[scale=0.8]{covering-stoyan20-20.mps} &
\includegraphics[scale=0.8]{covering-stoyan30-30.mps} &
\includegraphics[scale=0.8]{covering-stoyan40-40.mps} \\
(d) $m=20$ & (e) $m=30$ & (f) $m=40$ \\
\includegraphics[scale=0.8]{covering-stoyan50-50.mps} &
\includegraphics[scale=0.8]{covering-stoyan60-60.mps} &
\includegraphics[scale=0.8]{covering-stoyan70-70.mps} \\
(g) $m=50$ & (h) $m=60$ & (i) $m=70$ \\
\includegraphics[scale=0.8]{covering-stoyan80-80.mps} &
\includegraphics[scale=0.8]{covering-stoyan90-90.mps} &
\includegraphics[scale=0.8]{covering-stoyan100-100.mps} \\
(j) $m=80$ & (k) $m=90$ & (l) $m=100$ \\
\end{tabular}
\end{center}
\caption{(a) Non-convex polygon with holes considered in~\cite{stoyan}, partitioned into $p=14$ convex polygons as depicted in (b). Pictures from (c) to (l) display the solutions found with $m \in \{10,\dots,100\}$.} 
\label{fig:stoyan}
\end{figure}

\begin{figure}[ht!]
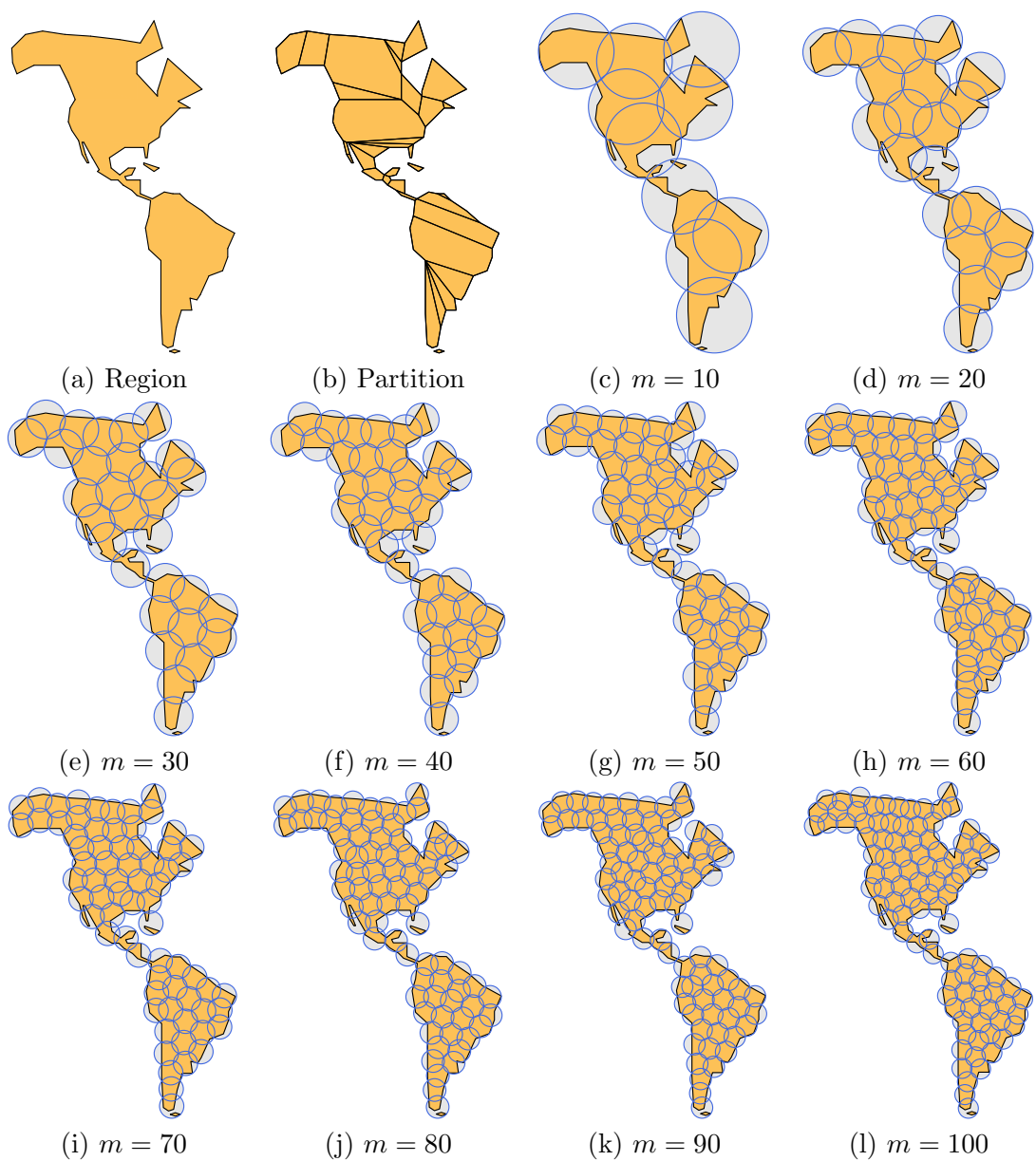

\begin{center}
\begin{tabular}{cccc}
\includegraphics[scale=0.6]{region-america.mps} &
\includegraphics[scale=0.6]{americapartition.mps} &
\includegraphics[scale=0.6]{covering-america10-10.mps} &
\includegraphics[scale=0.6]{covering-america20-20.mps} \\
(a) Region & (b) Partition & (c) $m=10$ & (d) $m=20$ \\
\includegraphics[scale=0.6]{covering-america30-30.mps} &
\includegraphics[scale=0.6]{covering-america40-40.mps} &
\includegraphics[scale=0.6]{covering-america50-50.mps} &
\includegraphics[scale=0.6]{covering-america60-60.mps} \\
(e) $m=30$ & (f) $m=40$ & (g) $m=50$ & (h) $m=60$ \\
\includegraphics[scale=0.6]{covering-america70-70.mps} &
\includegraphics[scale=0.6]{covering-america80-80.mps} &
\includegraphics[scale=0.6]{covering-america90-90.mps} &
\includegraphics[scale=0.6]{covering-america100-100.mps} \\
(i) $m=70$ & (j) $m=80$ & (k) $m=90$ & (l) $m=100$ \\
\end{tabular}
\end{center}
\caption{(a) Sketch of America available from~\cite[\S13.2]{bmbook} and already considered in~\cite{coveringfirst}, partitioned into $p=34$ convex polygons as depicted in (b). Pictures from (c) to (l) display the solutions found with $m \in \{10,\dots,100\}$.}
\label{fig:america}
\end{figure}

\begin{figure}[ht!]
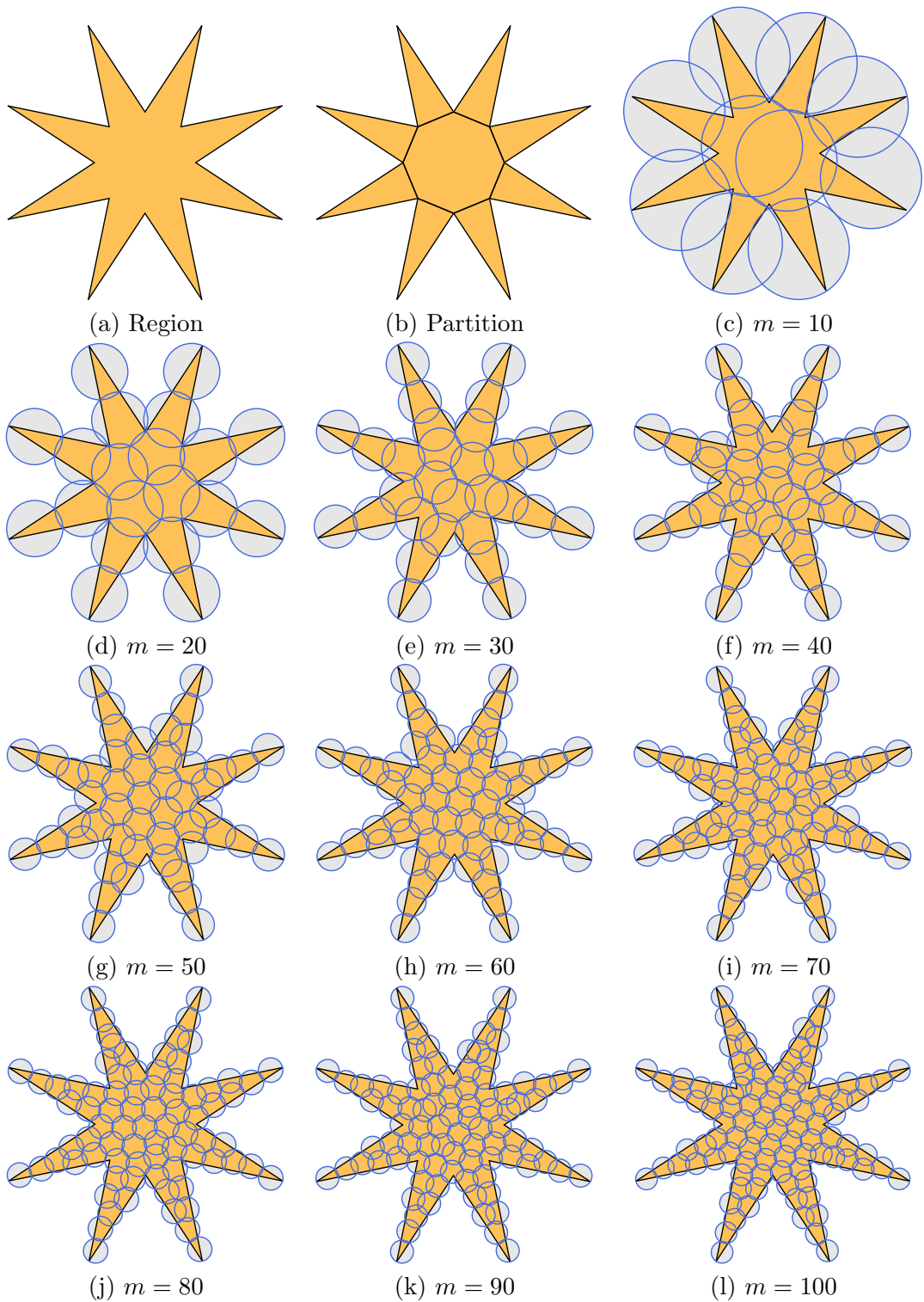

\begin{center}
\begin{tabular}{ccc}
\includegraphics[scale=0.75]{region-star.mps} &
\includegraphics[scale=0.75]{starpartition.mps} &
\includegraphics[scale=0.75]{covering-star10-10.mps} \\
(a) Region & (b) Partition & (c) $m=10$ \\
\includegraphics[scale=0.75]{covering-star20-20.mps} &
\includegraphics[scale=0.75]{covering-star30-30.mps} &
\includegraphics[scale=0.75]{covering-star40-40.mps} \\
(d) $m=20$ & (e) $m=30$ & (f) $m=40$ \\
\includegraphics[scale=0.75]{covering-star50-50.mps} &
\includegraphics[scale=0.75]{covering-star60-60.mps} &
\includegraphics[scale=0.75]{covering-star70-70.mps} \\
(g) $m=50$ & (h) $m=60$ & (i) $m=70$ \\
\includegraphics[scale=0.75]{covering-star80-80.mps} &
\includegraphics[scale=0.75]{covering-star90-90.mps} &
\includegraphics[scale=0.75]{covering-star100-100.mps} \\
(j) $m=80$ & (k) $m=90$ & (l) $m=100$ \\
\end{tabular}
\end{center}
\caption{(a) Eight-pointed star, partitioned into $p=9$ convex polygons as depicted in (b). Pictures from (c) to (l) display the solutions found with $m \in \{10,\dots,100\}$.} 
\label{fig:star}
\end{figure}

\begin{figure}[ht!]
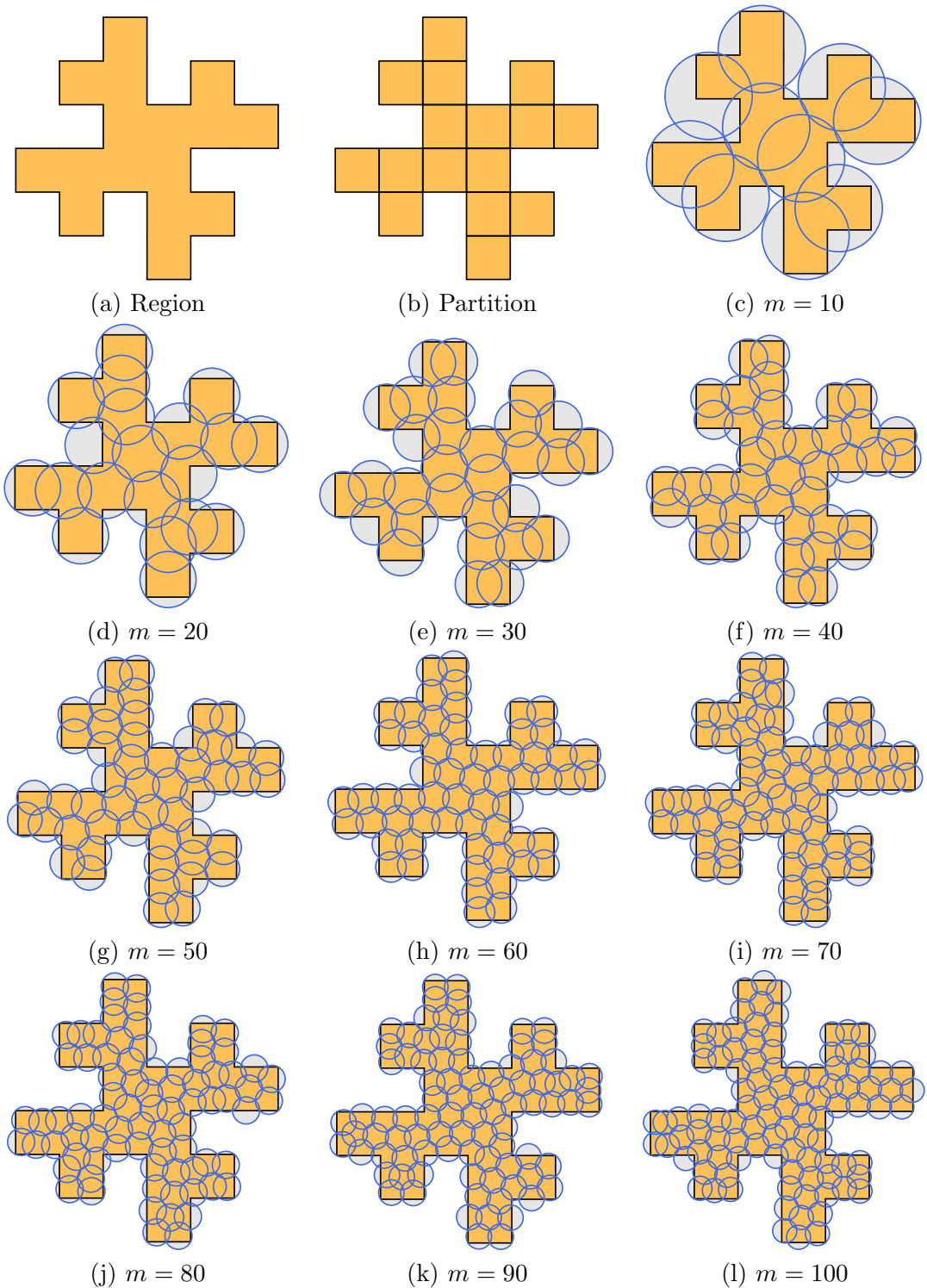

\begin{center}
\begin{tabular}{ccc}
\includegraphics[scale=0.9]{region-minkowski.mps} &
\includegraphics[scale=0.9]{minkowskipartition.mps} &
\includegraphics[scale=0.9]{covering-minkowski10-10.mps} \\
(a) Region & (b) Partition & (c) $m=10$ \\
\includegraphics[scale=0.9]{covering-minkowski20-20.mps} &
\includegraphics[scale=0.9]{covering-minkowski30-30.mps} &
\includegraphics[scale=0.9]{covering-minkowski40-40.mps} \\
(d) $m=20$ & (e) $m=30$ & (f) $m=40$ \\
\includegraphics[scale=0.9]{covering-minkowski50-50.mps} &
\includegraphics[scale=0.9]{covering-minkowski60-60.mps} &
\includegraphics[scale=0.9]{covering-minkowski70-70.mps} \\
(g) $m=50$ & (h) $m=60$ & (i) $m=70$ \\
\includegraphics[scale=0.9]{covering-minkowski80-80.mps} &
\includegraphics[scale=0.9]{covering-minkowski90-90.mps} &
\includegraphics[scale=0.9]{covering-minkowski100-100.mps} \\
(j) $m=80$ & (k) $m=90$ & (l) $m=100$ \\
\end{tabular}
\end{center}
\caption{(a) Minkowski island fractal, partitioned into $p=16$ convex polygons as depicted in (b). Pictures from (c) to (l) display the solutions found with $m \in \{10,\dots,100\}$.} 
\label{fig:minkowski}
\end{figure}

\begin{figure}[ht!]
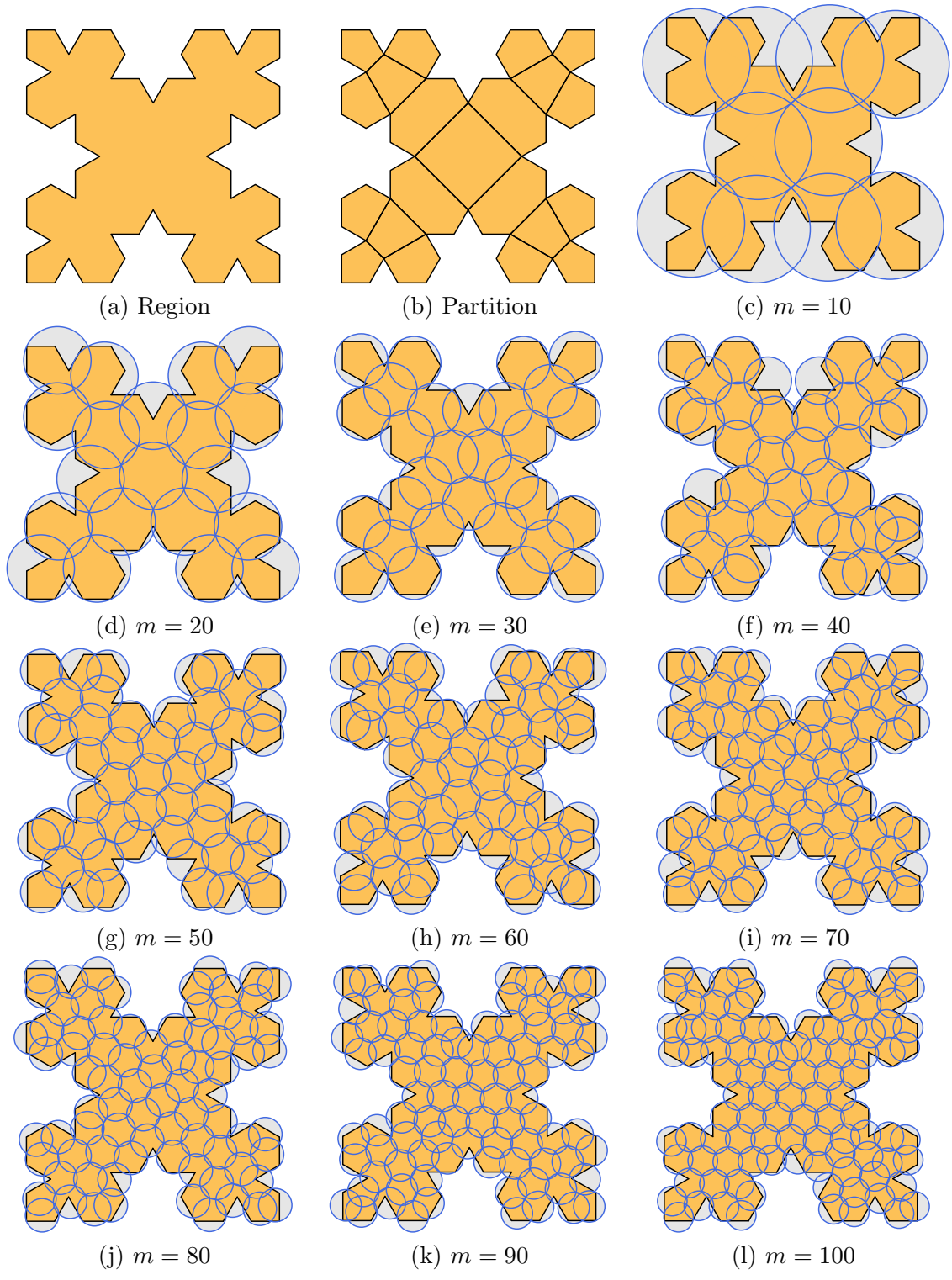

\begin{center}
\begin{tabular}{ccc}
\includegraphics[scale=0.8]{region-cesaro.mps} &
\includegraphics[scale=0.8]{cesaropartition.mps} &
\includegraphics[scale=0.8]{covering-cesaro10-10.mps} \\
(a) Region & (b) Partition & (c) $m=10$ \\
\includegraphics[scale=0.8]{covering-cesaro20-20.mps} &
\includegraphics[scale=0.8]{covering-cesaro30-30.mps} &
\includegraphics[scale=0.8]{covering-cesaro40-40.mps} \\
(d) $m=20$ & (e) $m=30$ & (f) $m=40$ \\
\includegraphics[scale=0.8]{covering-cesaro50-50.mps} &
\includegraphics[scale=0.8]{covering-cesaro60-60.mps} &
\includegraphics[scale=0.8]{covering-cesaro70-70.mps} \\
(g) $m=50$ & (h) $m=60$ & (i) $m=70$ \\
\includegraphics[scale=0.8]{covering-cesaro80-80.mps} &
\includegraphics[scale=0.8]{covering-cesaro90-90.mps} &
\includegraphics[scale=0.8]{covering-cesaro100-100.mps} \\
(j) $m=80$ & (k) $m=90$ & (l) $m=100$ \\
\end{tabular}
\end{center}
\caption{(a) Ces\`aro fractal, partitioned into $p=21$ convex polygons as depicted in (b). Pictures from (c) to (l) display the solutions found with $m \in \{10,\dots,100\}$.} 
\label{fig:cesaro}
\end{figure}

\section{Final considerations} \label{sec:conclusion}

From the shape optimization perspective, the present work completes \cite{coveringfirst} with a second-order shape sensitivity analysis for nonsmooth domains defined as a union of balls intersected with the domain to be covered.
The analysis of several singular cases seems to indicate that the assumptions used to derive $\nabla^2 G$ cannot be weakened. From the practical point of view, the exact calculation of $G$ and its first- and second-order derivatives represents the possibility, absent in~\cite{coveringfirst}, of solving very efficiently and with high accuracy, problems in which the area to be covered is given by a non-convex polygon.

We now discuss potential extensions of our approach.
Redefining $\Om(\bx,\boldsymbol{r}):=\cup_{i=1}^m B(x_i,r_i)$ and $G(\bx,\boldsymbol{r}):=\vol(A \setminus \Omega(\bx,\boldsymbol{r}))$, where $\boldsymbol{r}:=\{r_i\}_{i=1}^m$, expressions and algorithms to approximate $G(\bx,\boldsymbol{r})$, $\nabla G(\bx,\boldsymbol{r})$ and $\nabla^2 G(\bx,\boldsymbol{r})$ can be obtained with straightforward modifications to the introduced approach. From the practical point of view, underlying partitions that lead to exact calculations might be implemented using power diagrams~\cite{powerdia1,powerdia2}.
We observe that formulae (\ref{derivadasegunda}, \ref{d2Grr}, \ref{d2xixiG}, \ref{d2Gxixj}, \ref{eq:781}) are valid for general sets $A$ satisfying Assumptions~\ref{a1} and~\ref{a3}, but the exact numerical computation of $G$, $\nabla G$ and $\nabla^2 G$  requires $A$ to be a union of  non-overlapping convex polygons. 
The exact calculation of $\nabla G$ and $\nabla^2 G$ can actually be performed for any set $A$ such that the intersections of $\partial A$ with circles can be computed analytically.
However, the possibilities of computing $G$ exactly are more restricted as this requires the computation of integrals on subsets of $\partial A$.
In some specific cases, this calculation could be done exactly, for instance when $A$ is a union of balls.
Nevertheless, in more general cases the integrals on subsets of $\partial A$ could be efficiently approximated  with high accuracy.

The case where $\Om(\bx,r)$ is a union of objects with arbitrary (sufficiently smooth) shapes is challenging and would require a generalization of the techniques developed in \cite{coveringfirst} and in the present paper.
A key idea of our construction of the mappings $T_t$, which is still valid for objects with arbitrary shapes, is that the value of $T_t$
at the intersection points of the objects' boundaries (or the intersections with $\partial A$) is fully determined by the motion of these singular points, whereas the value of $T_t$ at the regular points of $\partial\Om(\bx,r)$ is underdetermined.
When the objects are balls, this underdetermination is conveniently resolved using polar coordinates to extend $T_t$ to the regular parts of $\partial\Om(\bx,r)$.
In the case of arbitrary shaped-objects however, a more general construction is required.
%
A generalization to three dimensions of the nonsmooth shape optimization techniques developed in \cite{coveringfirst} and in the present paper is conceivable but would also require a more general approach  to build  $T_t$.
Another interesting direction for future investigations would be the application of these techniques for optimization problems involving partial differential equations. The calculation of the shape derivatives would depend on the specific partial differential equation, but the construction of the transformations $T_t$ would remain the same.

\bibliographystyle{plain}
\bibliography{coveringsecond}

\appendix
\small

\section{Intersection of a convex polygon and a ball} \label{appen:alg4}

This appendix describes an algorithm that is an adaptation of a single iteration of the Sutherland-Hodgman algorithm~\cite{sutherland1974} to compute the intersection between a convex polygon and a ball. If the polygon has $\zeta$ vertices, a simple inspection of the algorithm makes it clear the algorithm has worst-case time complexity $\mathcal{O}(\zeta)$ and that the output is a curvilinear convex polygon (convex polygon whose sides are segments or circular arcs) with at most $2\zeta$ vertices.

\begin{algorithm}[ht!]
\caption{Intersection between a convex polygon $W$ and a closed ball $\overline{B(x,r)}$} \label{inter_pol_circle}
\KwInput{Ball given by radius $r$ and center $x$. Convex polygon $W$, given by a list of vertices $\{w_i\}_{i=1}^\zeta$, in counter-clockwise order.}
\KwOutput{Curvilinear polygon $S$, given by a list of tuples $\{(s_i, a_i)\}_{i=1}^\xi \in \{\mathbb{R}^2\times\{0, 1\}\}^\xi$, such that $(s_i, s_{(i~\mathrm{mod}~\xi) + 1})$ are the extreme points, in counter-clockwise order, of a line segment, if $a_i = 0$, or of an arc with center $x$, if $a_i = 1$. If $\overline{B(x,r)}\subset W$, then $S = \overline{B(x,r)}$ is represented by $\{\big((x_1, x_2 + r), 1\big)\}$.}
Set $S$ as an empty list; and denote by $S\frown{}s$ the action of adding $s$ to the end of $S$.\\
\For{$i=1, \dots, \zeta$}{
    $p \gets w_i$ and $q \gets w_{(i~\mathrm{mod}~\zeta) + 1}$\\
    \If{$p \in \overline{B(x,r)}$ and $q \in \overline{B(x,r)}$}{
        $S \gets S\frown (p, 0)$
    }
    \uElseIf{$p \in \overline{B(x,r)}$ and $q \notin \overline{B(x,r)}$}{
        $a \gets \overline{pq} \cap \partial{}B(x,r)$\\
        \If{$p \in B(x,r)$}{
            $S \gets S\frown (p, 0)$
        }
        $S \gets S \frown (a, 1)$
    }
    \uElseIf{$p \notin \overline{B(x,r)}$ and $q \in \overline{B(x,r)}$}{
        $a \gets \overline{pq} \cap \partial{}B(x,r)$\\
        $S \gets S\frown (a, 0)$
    }
    \uElseIf{$p \notin \overline{B(x,r)}$ and $q \notin \overline{B(x,r)}$}{
        \If{$| \, \overline{pq} \cap \partial{}B(x,r) | = 2$}{
            $S \gets S\frown (a, 0) \frown (b, 1)$, where $\overline{pq} \cap \partial{}B(x,r) = \{a, b\}$, ordered in the direction of $\overrightarrow{pq}$
        }
    }
}
\If{$|S| = 0$}{
    Set $S \gets \{\big((x_1, x_2 + r), 1\big)\}$.
}
\Return $S$
\end{algorithm}

\section{Problem data} \label{appen:probs}

In this appendix, we provide the description of the five problems illustrated in Figures~\ref{fig:stoyan}b--\ref{fig:cesaro}b and considered in the numerical experiments. The description of each problem consists in the list of the vertices, in counterclockwise order, of the convex polygons that compose the partition of the problem. The Fortran source code that describes the problems, as well as the source code to reproduce all numerical experiments, is available at \url{http://www.ime.usp.br/~egbirgin/}.

The non-convex polygon with holes shown in Figure~\ref{fig:stoyan}, with $\mathrm{Vol}(A) \approx 0.69111111111111101$, is composed by $p=14$ convex polygons. The vertices of polygons $A_1,\dots,A_{14}$ are the ones given below multiplied by~$\frac{1}{150}$:
\newline $\mathcal{V}(A_1) = \{(0,100), (0,70), (20,70), (20,100)\}$, 
\newline $\mathcal{V}(A_2) = \{(35,50), (20,70), (0,70), (0,30), (20,30)\}$,
\newline $\mathcal{V}(A_3) = \{(0,30), (0,-40), (20,-40), (20,30)\}$, 
\newline $\mathcal{V}(A_4) = \{(0,-40), (0,-50), (20,-50), (20,-40)\}$,
\newline $\mathcal{V}(A_5) = \{(20,100), (20,70), (70,70), (70,100)\}$,
\newline $\mathcal{V}(A_6) = \{(70,70), (80,70), (90,100), (70,100)\}$,
\newline $\mathcal{V}(A_7) = \{(80,70), (70,70), (45,50), (70,30), (90,50)\}$,
\newline $\mathcal{V}(A_8) = \{(70,30), (20,30), (20,-40), (60,0)\}$,
\newline $\mathcal{V}(A_9) = \{(110,20), (90,50), (70,30), (60,0), (90,0)\}$,
\newline $\mathcal{V}(A_{10}) = \{(130,-40), (130,-50), (150,-50), (150,-40)\}$,
\newline $\mathcal{V}(A_{11}) = \{(130,50), (110,20), (130,-40), (150,-40), (150,50)\}$, 
\newline $\mathcal{V}(A_{12}) = \{(130,100), (120,100), (110,80), (130,50), (150,50), (150,80)\}$,
\newline $\mathcal{V}(A_{13}) = \{(110,80), (120,100), (80,70), (90,50)\}$,
\newline $\mathcal{V}(A_{14}) = \{(110,20), (90,0), (130,-40)\}$.


The sketch of America shown in Figure~\ref{fig:america}, with $\mathrm{Vol}(A) \approx 0.17573124999999992$, is composed by $p=34$ convex polygons. The vertices of polygons $A_1,\dots,A_{34}$ are the ones given below multiplied by $\frac{1}{20}$:
\newline $\mathcal{V}(A_1) = \{(4.5,24), (3.5,23.8), (2.7,23), (2.75,22.15), (3,21.5), (4,22)\}$,
\newline $\mathcal{V}(A_2) = \{(4.5,24), (4,22), (5.5,22), (5.8,23.8)\}$, $\mathcal{V}(A_3) = \{(6,21), (6.4,20), (10,20)\}$,
\newline $\mathcal{V}(A_4) = \{(5.5,22), (6,21), (10,20), (10,21.5), (9,23.5), (7.3,23.7), (5.8,23.8)\}$,
\newline $\mathcal{V}(A_5) = \{(10,20), (11,19.1), (11.2,20.2), (10,21.5)\}$, 
\newline $\mathcal{V}(A_6) = \{(10,21.5), (10,22.2), (2.7,23), (9,23.5)\}$,
\newline $\mathcal{V}(A_7) = \{(10,22.2), (10.2,23.3), (9,23.5)\}$,
\newline $\mathcal{V}(A_8) = \{(10,22.2), (11.5,23), (11,24.6), (10.2,23.3)\}$,
\newline $\mathcal{V}(A_9) = \{(11,19.1), (11.4,18.4), (12.5,19.5), (12.4,19.9), (11.2,20.2)\}$,
\newline $\mathcal{V}(A_{10}) = \{(12.4,19.9), (13.8,20.6), (11.8,22.4), (11.2,20.5), (11.2,20.2)\}$,
\newline $\mathcal{V}(A_{11}) = \{(12.5,19.5), (13.1,19.5), (12.4,19.9)\}$,
\newline $\mathcal{V}(A_{12}) = \{(6.4,20), (6.1,19.5), (6,18.7), (6.2,18.2), (6.6,17.6), (6.8,17.5), (6.9,17.5), (11.3,17.8), \newline (11.4,18.4), (11,19.1), (10,20)\}$,
\newline $\mathcal{V}(A_{13}) = \{(6.9,17.5), (10.7,17.4), (11.3,17.8)\}$,
\newline $\mathcal{V}(A_{14}) = \{(10.4,17.2), (10.5,16.6), (10.6,16.6), (10.7,17.4)\}$,
\newline $\mathcal{V}(A_{15}) = \{(6.9,17.5), (9.3,17.2), (10.4,17.2), (10.7,17.4)\}$, 
\newline $\mathcal{V}(A_{16}) = \{(6.9,17.5), (8.4,16.6), (9.3,17.2)\}$,
\newline $\mathcal{V}(A_{17}) = \{(6.9,17.5), (7.4,16.6), (7.8,15.9), (8.5,16), (8.4,16.6)\}$,
\newline $\mathcal{V}(A_{18}) = \{(7.8,15.9), (7.7,15.8), (8.5,15.3), (8.9,15.3), (9,15.6), (8.5,16)\}$,
\newline $\mathcal{V}(A_{19}) = \{(8.9,15.3), (9.2,15), (9.4,15.3), (9.3,15.5), (9,15.6)\}$,
\newline $\mathcal{V}(A_{20}) = \{(9.3,15.5), (9.7,15.6), (9.9,16), (9.5,16), (9,15.6)\}$,
\newline $\mathcal{V}(A_{21}) = \{(6.6,17.6), (6.8,16.8), (7,16.8), (6.8,17.5)\}$,
\newline $\mathcal{V}(A_{22}) = \{(6.8,16.8), (7.1,16.3), (7.2,16.3), (7,16.8)\}$,
\newline $\mathcal{V}(A_{23}) = \{(9.2,15), (9.7,14.7), (10.2,14.5), (10.2,15.3), (9.4,15.3)\}$,
\newline $\mathcal{V}(A_{24}) = \{(9.7,14.7), (10,14.4), (10.8,14.1), (10.9,14.2), (10.2,14.5)\}$,
\newline $\mathcal{V}(A_{25}) = \{(10.4,16.2), (11,15.8), (11.3,16), (10.4,16.3)\}$,
\newline $\mathcal{V}(A_{26}) = \{(10.7,13.2), (10.5,12.5), (10.7,11.25), (11.4,10.6), (14.2,9.7), (15,10), (15.3,10.8), \newline (15.3,11.3)\}$,
\newline $\mathcal{V}(A_{27}) = \{(12.2,5.4), (11.9,5.3), (12.2,5.2), (12.2,5.4)\}$,
\newline $\mathcal{V}(A_{28}) = \{(15.3,11.3), (15.7,12.2), (14.6,12.8), (10.9,14.2), (10.8,14.1), (10.7,13.2)\}$,
\newline $\mathcal{V}(A_{29}) = \{(14.6,12.8), (13.8,13.5), (12.9,14.1), (12.1,14.5), (11.6,14.6), (10.9,14.2)\}$,
\newline $\mathcal{V}(A_{30}) = \{(12.9,14.1), (12.5,14.5), (12.1,14.5)\}$,
\newline $\mathcal{V}(A_{31}) = \{(11.4,10.6), (11.4,7.5), (11.5,5.7), (11.8,5.5), (12.1,5.6), (12.3,6.7)\}$,
\newline $\mathcal{V}(A_{32}) = \{(12.3,6.7), (12.6,7.7), (11.4,10.6)\}$,
\newline $\mathcal{V}(A_{33}) = \{(12.6,7.7), (13.2,7.7), (13.1,8.4), (11.4,10.6)\}$,
\newline $\mathcal{V}(A_{34}) = \{(13.1,8.4), (13.5,8.3), (13.7,8.6), (14.2,9.7), (11.4,10.6)\}$. 

The star shape shown in Figure~\ref{fig:star} is composed by $p=9$ convex polygons, namely, a regular octagon and eight isosceles triangles with height equal to twice the radius of the circumscribed circle to the octagon. The octagon, named $A_1$, is centered at the origin and its sides have length equal to one. Denote by $R$ the radius of the circumscribed circle to the octagon, which is given by $R = 1/(2 \sin(\pi/8))$. The vertices of the octagon are then given by $\mathcal{V}(A_1) = \{(R\cos(k \pi/4), R\sin(k \pi/4))\}_{k=1}^{8}$. The height of the isosceles triangles, which we denote by $A_2,\dots,A_9$, is equal to $2R$. Let $d = ([\frac{1}{2}(R \cos (\pi/4) +R \cos (\pi/2))]^2 + [\frac{1}{2}(R \sin (\pi/4) +R \sin (\pi/2))]^2)^{1/2}$ be the distance of the origin to the middle point of any edge of the octagon; and let $d' = d + 2R$.  The vertices of $A_2$ are given by $\mathcal{V}(A_2) = \{(R\cos(\pi/4), R\sin(\pi/4)), (R\cos(2\pi), R\sin(2\pi)), (d'\cos(\pi/8), d'\sin(\pi/8))\}$. The vertices of $A_i$, for $i=3,\dots,9$, are given by a rotation of $\pi/4$ applied to the vertices of $A_{i-1}$. The area of $A$ is given by
\begin{equation*}
\mathrm{Vol}(A) = \sum_{j=1}^{9} \mathrm{Vol}(A_j) = 2(1+\sqrt{2}) + 8R \approx 15.28093084375720.
\end{equation*}


The Minkowski island fractal shown in Figure~\ref{fig:minkowski}, with $\mathrm{Vol}(A) = 16$, is composed by $p=16$ unit squares. Each square can be represented by its bottom-left corners $\mathcal{V} = \{(3,0),(1,1), (3,1), (4,1)$, $(0,2), (1,2), (2,2)$, $(3,2), (2,3), (3,3), (4,3), (5,3), (1,4), (2,4), (4,4), (2,5)\}$.

The Cesàro fractal shown in Figure~\ref{fig:cesaro}, with $\mathrm{Vol}(A) = 0.72201653705684687$, is composed by $p=21$ convex polygons. It can be seen that this partition is composed by four symmetrical groups of convex polygons, in addition to a central square. We display here the vertices of the central square, namely $A_1$, and the vertices of the polygons in the bottom-left group, namely $A_2,\dots,A_6$. The vertices of the polygons of the other three groups can be obtained by rotating, around $(0.5, 0.5)$, an angle of $\pi/2$, $\pi$ and $3\pi/2$, respectively. The vertices of polygons $A_1,\dots,A_6$ are the ones given below multiplied by $\frac{1}{18}$:
\newline $\mathcal{V}(A_1) = \{(9,3\sqrt{3}), (18-3\sqrt{3},9), (9,18-3\sqrt{3}), (3\sqrt{3},9)\}$,
\newline $\mathcal{V}(A_2) = \{(0,0), (2,0), (3,\sqrt{3}), (\sqrt{3},3), (0,2)\}$,
\newline $\mathcal{V}(A_3) = \{(4,0), (6,0), (7,\sqrt{3}), (6,2\sqrt{3}), (3,\sqrt{3})\}$,
\newline $\mathcal{V}(A_4) = \{(6,2\sqrt{3}), (8,2\sqrt{3}), (9,3\sqrt{3}), (3\sqrt{3},9), (2\sqrt{3},8), (2\sqrt{3},6)\}$,
\newline $\mathcal{V}(A_5) = \{(0,6), (0,4), (\sqrt{3},3), (2\sqrt{3},6), (\sqrt{3},7)\}$,
\newline $\mathcal{V}(A_6) = \{(3,\sqrt{3}), (6,2\sqrt{3}), (2\sqrt{3},6), (\sqrt{3},3)\}$.


\end{document}